\documentclass[11pt]{article}

\usepackage{amsfonts}
\usepackage{amsmath}
\usepackage{amssymb}
\usepackage{epsfig}
\usepackage{bm}
\usepackage{enumerate}
\numberwithin{equation}{section} 

\title{Well-posed boundary integral equation formulations and Nystr\"om discretizations for the solution of Helmholtz transmission problems in two-dimensional Lipschitz domains}
\author{V\'{\i}ctor Dom\'{\i}nguez\thanks{Dep. Ingenier\'{i}a Matem\'atica e Inform\'atica, Universidad P\'{u}blica de Navarra. Campus de Tudela 31500 - Tudela, Spain, e-mail: victor.dominguez@unavarra.es.}, Mark Lyon\thanks{Department of Mathematics and Statistics, University of New Hampshire, Durham, NH 03861, USA
e-mail: mark.lyon@unh.edu.}, Catalin Turc\thanks{  Department of
Mathematical Sciences and Center for Applied Mathematics and Statistics, New Jersey  Institute of Technology,
Univ. Heights. 323 Dr. M. L. King Jr. Blvd, Newark, NJ 07102, USA, e-mail: catalin.c.turc@njit.edu.}}

\newtheorem{theorem}{Theorem}[section]
\newtheorem{lemma}[theorem]{Lemma}

\newtheorem{remark}[theorem]{Remark}
\newenvironment{proof}{\hspace{0.5cm} {\bf Proof.}}
{$\quad {}_\blacksquare$\vspace{0.3cm}}

\date{}
\newcommand{\triple}[1]{{\left\vert\kern-0.25ex\left\vert\kern-0.25ex\left\vert #1 
    \right\vert\kern-0.25ex\right\vert\kern-0.25ex\right\vert}}

\begin{document}
\maketitle
\begin{abstract}
  We present a comparison between the performance of solvers based on Nystr\"om discretizations of several well-posed boundary integral equation formulations of Helmholtz transmission problems in two-dimensional Lipschitz domains. Specifically, we focus on the following four classes of boundary integral formulations of Helmholtz transmission problems (1) the classical first kind integral equations for transmission problems~\cite{costabel-stephan}, (2) the classical second kind integral equations for transmission problems~\cite{KressRoach}, (3) the {\em single} integral equation formulations~\cite{KleinmanMartin}, and (4) certain {\em direct} counterparts of recently introduced Generalized Combined Source Integral Equations~\cite{turc2,turc3}.  The former two formulations were the only formulations whose well-posedness in Lipschitz domains was rigorously established~\cite{costabel-stephan,ToWe:1993}. We establish the well-posedness of the latter two formulations in appropriate functional spaces of boundary traces of solutions of transmission Helmholtz problems in Lipschitz domains. We give ample numerical evidence that Nystr\"om solvers based on formulations (3) and (4) are computationally more advantageous than solvers based on the classical formulations (1) and (2), especially in the case of high-contrast transmission problems at high frequencies.   
   
\end{abstract}

\section{Introduction\label{intro}}

\parskip 2pt plus2pt minus1pt

A wide variety of well-posed boundary integral equations for the solution of Helmholtz transmission problems has been proposed in the literature, at least in the case when the interfaces of material discontinuity are regular enough. Most of these formulations are derived from representations of the fields in each region filled by a homogeneous material by suitable combinations of single and double layer potentials. The enforcement of the continuity of solutions and their normal derivatives across interfaces of material discontinuity leads to Combined Field Integral Equations (CFIE) of transmission scattering problems. Some of these integral formulations involve two unknowns \cite{costabel-stephan,KittapaKleinman,KressRoach,MR2222969,rokhlin-dielectric}, three or more unknowns \cite{LaRaSa:2009,MR2391686}, 
while others involve one unknown per each interface of material discontinuity~\cite{KleinmanMartin}. It is also possible to formulate transmission problems in terms of both interior and exterior traces--Multi-Trace Formulations (MTF), that is using four unknowns per each interface of material discontinuity~\cite{jerez-hanckes1,jerez-hanckes2}.  More general boundary problems, which include not only transmission conditions but mixed Dirichet-Neumann conditions in complex geometrical configurations as well, have been studied in \cite{MR984053}. 

In the technologically important case of transmission problems at high-frequencies, the numerical solutions of boundary integral equation formulations typically rely on Krylov subspace iterative methods. As in the case of impenetrable scattering problems~\cite{turc1}, the classical boundary integral equations of transmission problems are not particularly well suited for Krylov iterative solutions of transmission problems at high-frequencies. We have demonstrated recently that a novel class of boundary integral equations referred to as Generalized Combined Source Integral Equations (GCSIE)~\cite{turc2,turc3} is a more favorable alternative for smooth transmission problems that involve high-contrast configurations at high-frequencies. The main scope of this paper is to investigate to what extent the aforementioned claim is valid in the case of high-frequency, high-contrast Helmholtz transmission problems when the interface of material continuity is a Lipschitz curve. 

An important question related to boundary integral equation (BIE) formulations of linear, constant-coefficient PDEs is whether the BIE are well-posed. This issue is typically settled via Fredholm arguments whose flavor differ significantly from the case of regular boundaries to the case of Lipschitz boundaries. The case of smooth boundaries is extremely well understood and researched, as one can take full advantage of the increased smoothing properties of double layer operators that guarantee compactness properties needed in the Fredholm theory~\cite{turc3}. In addition, a very general methodology based on coercive approximations of Dirichlet-to-Neumann operators can deliver optimally conditioned boundary integral formulations (GCSIE) for transmission Helmholtz problems~\cite{turc3}; the aforementioned enhanced smoothing properties play a major role in  establishing the well-posedness of the GCSIE in the smooth case. We address in this paper the issues of well-posedness and well-conditioning of the GCSIE formulations in the Lipschitz case by making use of deep results from harmonic analysis~\cite{Ver:1984} and certain duality pairings. We also present a {\em direct} counterpart to the GCSIE formulations which we refer to as Regularized Combined Field Integral Equations (CFIER). From the numerical point of view, the main advantage of the direct CFIER formulations is given by the fact that they employ as unknowns Dirichlet and Neumann traces of transmission problems, whose singular behavior around corner points is well understood~\cite{costabel-stephan} and thus can be resolved by graded meshes towards corners. This brings us to the second major point of this paper, high-order Nystr\"om discretizations of transmission boundary integral equations in two-dimensional Lipschitz domains. 

High-order Nystr\"om methods typically employ graded meshes in order to deal with singularities associated with solutions of boundary integral equations in domains with corners~\cite{KressCorner,helsing1,helsing2}. One issue that arises in this regard is the possibly unbounded nature of such solutions in the neighborhood of corners in the case of integral formulations of the second kind. While in several instances the issue can be avoided by resorting to alternative integral equation formulations~\cite{turc_corner_N} whose solutions are regular enough (e.g. H\"older continuous), in many others, including the case of transmission problems in domains with corners, the unboundedness of solutions cannot be avoided. Two main approaches to tackle the unbounded nature of solutions of integral equations of the second kind have been recently introduced: (a) one that relies on incorporation of the {\em known} asymptotic infinite behavior of solutions in the vicinity of corners and exact cancellations of infinite quantities~\cite{turc_ovall}; and (b) one that uses jacobians associated with graded meshes in order to introduce more regular {\em weighted} solutions as new unknowns of newly {\em weighted} integral formulations of the second kind~\cite{monegato,bremer,greengard1,greengard2,LuW}. We pursue a version of the latter approach in this paper in conjunction with simple modifications of a Nystr\"om method based on global trigonometric interpolation, singular kernel-splitting, and analytic evaluations of integrals that involve products of certain singular functions and Fourier harmonics~\cite{kusmaul,martensen}. Several ideas in this approach were introduced in~\cite{LuW} for the discretization of the first kind boundary integral equation formulation of transmission problems~\cite{costabel-stephan} in domains with corners. The method incorporates sigmoid transforms~\cite{KressCorner} within parametrizations of domains with corners and it uses the Jacobians of these transformations as multiplicative weights to define new unknowns. Specifically, the focus of the paper is on direct integral formulations of transmission problems whereby the unknowns are the Dirichlet and Neumann traces of solutions of transmission problems on the Lipschitz boundary. A weighted Neumann trace defined as the product of the derivatives of the sigmoid parametrizations and the usual Neumann trace of solution of transmission problems is introduced as a new unknown; given that the derivatives of the parametrizations that incorporate sigmoid transforms vanish polynomially at corners, the weighted traces are more regular for large enough values of the order of the polynomial in the sigmoid transform. Introducing new weighted unknowns also require definition of new weighted boundary integral equations that involve weighted versions of the four scattering boundary integral operators. It turns out that the kernel splitting techniques originally developed for smooth curves~\cite{KressH} can be easily extended to the weighted boundary integral operators, delivering a high-order Nystr\"om discretizations for the various formulations considered in this paper. We give ample numerical evidence that in the high-contrast, high-frequency regime the single integral equation formulations and our novel regularized formulations have superior spectral properties over the classical formulations of transmission problems, giving rise to important computational savings.

The paper is organized as follows: in Section~\ref{cfie} we formulate the Helmholtz transmission problem we are interested in; in Section~\ref{di_ind_cfie} we recount the definition of the four scattering boundary integral operators and we discuss several boundary integral formulations of the Helmholtz transmission problem; in Section~\ref{well-posedness} we establish the well-posedness of several of the boundary integral formulations discussed in this paper, including the regularized formulations CFIER; finally, in Section~\ref{singular_int} we present high-order Nystr\"om discretizations of the various boundary integral equations considered in this paper and we carry out a comparison between solvers based on this formulations that emphasizes the benefits that can be garnered from the use of single integral equations and the regularized integral equations.

\section{Integral Equations of Helmholtz transmission problems\label{cfie}}

We consider the problem of evaluating the time-harmonic fields $u^1$ and $u^2$ that result as an incident field $u^{inc}$ impinges upon the boundary
$\Gamma$ of a homogeneous dielectric scatterer $D_2$ which occupies a bounded region in $\mathbb{R}^2$. We assume that both media occupying $D_2$ and its exterior are nonmagnetic, and the electric permitivity of the dielectric material inside the domain $D_2$ is denoted by $\epsilon_2$ while that of the medium occupying the exterior of $D_2$ is denoted by $\epsilon_1$. The frequency domain dielectric transmission problem is formulated in terms of finding fields $u^1$ and $u^2$ that are solutions to the Helmholtz equations
\begin{equation} 
  \label{eq:Ac_i}
\begin{aligned}
  \Delta u^2+k_2^2 u^2&=& 0, \qquad &\mathrm{in}\ D_2,\\
  \Delta u^1+k_1^2 u^1&=&0,\qquad &\mathrm{in}\ D_1=\mathbb{R}^2\setminus {\overline{D_2}},
\end{aligned}
\end{equation} 
where the wavenumbers $k_i,i=1,2$ are defined as $k_i=\omega\sqrt{\epsilon_i}, i=1,2$ in terms of the frequency $\omega$. The incident field $u^{inc}$ is assumed to satisfy 
\begin{equation}
  \label{eq:Maxwell_inc}
  \Delta u^{inc}+k_1^2 u^{inc}=0 \qquad \mathrm{in}\  \widetilde{D}_{2},
\end{equation}
where $\widetilde{D}_{2}$ is an open neighborhood of  $ \overline{D_2}$.  Therefore  $u^{inc}$ has to be smooth, actually analytic, in $\overline{D_2}$, which includes  plane waves or spherical waves from a point source placed in the exterior domain ${D_1}$.

In addition, the fields $u^{1}$, $u^{inc}$, and $u^2$ are related on the boundary $\Gamma$ by the the following boundary conditions
\begin{eqnarray}
\label{eq:bc}
\gamma_D^1 u^1 + {\gamma_D u^{inc}} &=&\gamma_D^2 u^2\qquad \rm{on}\ \Gamma \nonumber\\
\gamma_N^1 u^1 + {\gamma_N u^{inc}}&=&\rho \gamma_N^2u^2\qquad \rm{on}\ \Gamma
\end{eqnarray} 
with $\rho>0$.  
In equations~\eqref{eq:bc} and what follows $\gamma_D^i,i=1,2$ denote  exterior ($i=1$) and interior Dirichlet traces ($i=2$). Similarly $\gamma_N^i,i=1,2$ denote exterior and interior Neumann traces taken with respect to the exterior unit normal on $\Gamma$. When  both  Dirichlet, respectively Neumann, traces coincide, we will simply write $\gamma_D$, respectively $\gamma_N$.  (Notice that $\gamma_{D}^1 u^{inc}=\gamma_{D}^2 u^{inc}$,
$\gamma_{N}^1 u^{inc}=\gamma_{N}^2 u^{inc}$ and therefore we are allowed to use  $\gamma_{\{D,N\}} u^{inc}$ in \eqref{eq:bc}).

We assume in what follows that the boundary $\Gamma$ is a closed Lipschitz curve in $\mathbb{R}^2$. Depending on the type of scattering problem, the transmission coefficient ${\rho}$ in equations~\eqref{eq:bc} can be either $1$ (E-polarized) or $\epsilon_1/\epsilon_2$ (H-polarized). We furthermore require that $u^1$ satisfies the Sommerfeld radiation conditions at infinity:
\begin{equation}\label{eq:radiation}
\lim_{|\bm{r}|\to\infty}|\bm{r}|^{1/2}(\partial u^1/\partial \bm{r} - ik_1u^1)=0.
\end{equation}
(Here $\partial  /\partial \bm{r}$  is the radial derivative).
Note that under these assumptions the  wavenumbers $k_i,$ $i=1,2$ are  real numbers. It is well known that in this case the systems of partial differential equations~\eqref{eq:Ac_i}-\eqref{eq:Maxwell_inc} together with the boundary conditions~\eqref{eq:bc} and the radiation condition~\eqref{eq:radiation} has a unique solution~\cite{KressRoach,KleinmanMartin}. The results in this text can be extended to the case of complex wavenumbers $k_i,i=1,2$, provided we assume uniqueness of the transmission problem and its adjoint, that is, for the same transmission problem but with wavenumbers $k_1$ for $D_2$ and $k_2$ for $D_1$ \cite{KressRoach}.




\section{Boundary integral formulation for the 
 transmission problems\label{di_ind_cfie}}

A variety of well-posed integral equations for the transmission problem~\eqref{eq:Ac_i}-\eqref{eq:bc} exist~\cite{KressRoach,costabel-stephan,KleinmanMartin,turc2}. On one hand, integral equations formulations for transmission problems can be formulated as a $2\times 2$ system of integral equations which can be derived from (a) Green's formulas in both domains $D_1$ and $D_2$, in which case they are referred to as {\em direct} integral equation formulations~\cite{costabel-stephan,KleinmanMartin}, (b) from representations of the fields $u^j,j=1,2$ in forms of suitable combinations of single and double layer potentials in both domains $D_1$ and $D_2$, in which case they are referred to as {\em indirect} integral equation formulations~\cite{KressRoach}, or (c) from Green's formulas and suitable approximations to exterior and interior Dirichlet-to-Neumann operators, in which case they are referred to as {\em regularized} combined field integral equations or generalized combined source integral equations~\cite{turc2}. On the other hand, integral equations formulations for transmission problems can be formulated as {\em single} integral equations which can be derived from (d) Green's formulas in one of the domains and (indirect) combined field representations in the other domain~\cite{KleinmanMartin}. The strategies recounted above lead to Fredholm second kind boundary integral equations for the solution of transmission problems~\cite{KressRoach,KleinmanMartin,turc2,turc3}, at least in the case when the curve $\Gamma$ is smooth enough ($C^3$ suffices). The first part of this paper is devoted to establishing the well-posedness of the boundary integral equations of the type (c) and (d) in the case when the curve $\Gamma$ is Lipschitz. To this end, we begin by reviewing the definition and mapping properties of the various scattering boundary integral operators. 

\subsection{Layer potentials and operators}
We start with the definition of the single and double layer potentials. Given a wavenumber $k$ such that $\Re{k}>0$ and $\Im{k}\geq 0$, and a density $\varphi$ defined on $\Gamma$, we define the single layer potential as
$$[SL_k(\varphi)](\mathbf{z}):=\int_\Gamma G_k(\mathbf{z}-\mathbf{y})\varphi(\mathbf{y})ds(\mathbf{y}),\ \mathbf{z}\in\mathbb{R}^2\setminus\Gamma$$
and the double layer potential as
$$[DL_k(\varphi)](\mathbf{z}):=\int_\Gamma \frac{\partial G_k(\mathbf{z}-\mathbf{y})}{\partial\mathbf{n}(\mathbf{y})}\varphi(\mathbf{y})ds(\mathbf{y}),\ \mathbf{z}\in\mathbb{R}^2\setminus\Gamma$$
where $G_k(\mathbf{x})=\frac{i}{4}H_0^{(1)}(k|\mathbf{x}|)$ represents the two-dimensional Green's function of the Helmholtz equation with wavenumber $k$. The Dirichlet and Neumann exterior and interior traces on $\Gamma$ of the single and double layer potentials corresponding to the wavenumber $k$ and a density $\varphi$ are given by
\begin{equation}\label{traces}
\begin{split}
\gamma_D^1 SL_k(\varphi)&=\gamma_D^2 SL_k(\varphi)= {\gamma_D SL_k(\varphi)}=S_k\varphi \\
\gamma_N^j SL_k(\varphi)&=(-1)^j\frac{\varphi}{2}+K_k^\top \varphi\quad j=1,2\\
\gamma_D^j DL_k(\varphi)&=(-1)^{j+1}\frac{\varphi}{2}+K_k\varphi\quad j=1,2\\
\gamma_N^1 DL_k(\varphi)&=\gamma_N^2 DL_k(\varphi)= {\gamma_N DL_k(\varphi)}=\varphi.
\end{split}
\end{equation}
In equations~\eqref{traces} the operators $K_k$ and $K^\top_k$, usually
referred to as double and adjoint double layer operators, are defined for a given wavenumber $k$ and density $\varphi$ as
\begin{equation}
\label{eq:double}
(K_k\varphi)(\mathbf x):=\int_{\Gamma}\frac{\partial G_k(\mathbf x-\mathbf y)}{\partial\mathbf{n}(\mathbf y)}\varphi(\mathbf y)ds(\mathbf y),\ \mathbf x\ {\rm on}\ \Gamma
\end{equation}
and 
\begin{equation}
\label{eq:adj_double}
(K_k^\top\varphi)(\mathbf x):=\int_{\Gamma}\frac{\partial G_k(\mathbf x-\mathbf y)}{\partial\mathbf{n}(\mathbf x)}\varphi(\mathbf y)ds(\mathbf y),\ \mathbf x\ {\rm on}\ \Gamma.
\end{equation}
Furthermore, for a given wavenumber $k$ and density $\varphi$, 
the operator $N_k$ denotes the Neumann trace of the double layer potential on 
$\Gamma$ given in terms of a Hadamard Finite Part (FP)  integral which can be re-expressed in terms of a Cauchy Principal Value (PV) integral that involves the tangential derivative $\partial_s$ on the curve $\Gamma$
\begin{eqnarray}
 \label{eq:normal_double:0} 
(N_k \varphi)(\mathbf x) &:=& \text{FP} \int_\Gamma \frac{\partial^{2}G_k(\mathbf x -\mathbf y)}{\partial \mathbf{n}(\mathbf x) \partial \mathbf{n}(\mathbf y)} \varphi(\mathbf y)ds(\mathbf y)\nonumber\\
 \label{eq:normal_double} 
&=&k^{2}\int_\Gamma G_k(\mathbf x -\mathbf y)
(\mathbf{n}(\mathbf x)\cdot\mathbf{n}(\mathbf y))\varphi(\mathbf y)ds(\mathbf y)+ {\rm PV}
\int_\Gamma \partial_s G_k(\mathbf x -\mathbf y)\partial_s \varphi(\mathbf y)ds(\mathbf y).
\quad
\end{eqnarray}
Finally, the single layer operator $S_k$ is defined for a wavenumber $k$ as
\begin{equation}\label{eq:sl}
(S_k\varphi)(\mathbf x):=\int_\Gamma G_k(\mathbf x -\mathbf y)\varphi(\mathbf y)ds(\mathbf y),\ \mathbf{x}\ {\rm on} \ \Gamma
\end{equation} 
for a density function $\varphi$ defined on $\Gamma$. 

We commit here another slight abuse of notation and denote in what follows $
 SL_1,  DL_1, S_1,\  N_1 \ldots, $
for layer potentials and operators corresponding to $k_1$. That is,
\[
 SL_{1}=SL_{k_1}\quad 
 DL_{1}=DL_{k_1},\quad S_1=S_{k_1},\quad \text{etc.}. 
\]
Analogously
\[
 SL_{2}=SL_{k_2}\quad 
 DL_{2}=DL_{k_2},\quad S_2=S_{k_2},\quad \text{etc.}. 
\]
We stress that the context will avoid any possible confusion between  indices and wavenumbers. 
This convention helps us to enhance and lighten the notation. 
Notice that Green identities can be now written in the simple form: 
\begin{equation}\label{eq:Green}
 u^j=(-1)^j SL_j(\gamma_N^j u^j)-(-1)^j DL_j(\gamma_D^j u^j). 
\end{equation} 
Similarly, 
\begin{equation}\label{eq:C1}
 C_j=\tfrac12\begin{bmatrix}I\\ &I\end{bmatrix}+(-1)^j\begin{bmatrix}
                           -K_j & S_j\\
                           -N_j& K_j^\top
                          \end{bmatrix}, \quad 
                          j=1,2
\end{equation}
are the  exterior/interior  Calder\'{o}n  projections associated to the exterior/interior Helmholtz equation: 
\begin{equation}\label{eq:C2}
 C_j^2=C_j,\quad C_j\begin{bmatrix}
                     \gamma_D^j u^j\\
                     \gamma_N^j u^j
                    \end{bmatrix}=
\begin{bmatrix}
                     \gamma_D^j u^j\\
                     \gamma_N^j u^j
                    \end{bmatrix}.
\end{equation}
We recall that from \eqref{eq:C1}-\eqref{eq:C2} one deduces easily 
\begin{equation}\label{eq:calderon}
 S_kN_k=-\frac{1}{4}I + K_k^2,\quad 
 N_kS_k=-\frac{1}{4}I + (K_k^\top)^2,\quad 
 N_k K_k= K_k^\top N_k ,\quad K_kS_k=S_kK_k^\top .
\end{equation}

We end this section noticing that since $u^{inc}$ solves $\Delta v +k_1^2 v=0$ in $D_2$,  it holds 
 \begin{equation}\label{eq:Cinc}
C_1\begin{bmatrix}
                     \gamma_D u^{inc}\\
                     \gamma_N u^{ inc}
                    \end{bmatrix}={\bf 0}.
\end{equation}
and therefore
\begin{equation}\label{eq:Cinc2}
 -SL_1(\gamma_N u^{inc})+DL_1(\gamma_D u^{inc})=0,\quad \text{in }D_1
\end{equation}


\subsection{Boundary integral equations}

Let us introduce first the total field given by
\begin{equation}\label{eq:totalwave}
u^t=\begin{cases}
     u^{1}+u^{inc},&\text{in }D_1\\
     u^{2}&\text{in }D_2
    \end{cases}
\end{equation}
The unknowns in the direct formulations we consider in this paper are 
\begin{equation}
 \label{eq:unknowns}
 \gamma_D u^t =\gamma^1_D(u^1+u^{inc})=\gamma^2_D u^2,\quad  \gamma^1_N u^t=\gamma^1_N(u^1+u^{inc})
\end{equation}
the Dirichet and Neumann data for the total field from the unbounded domain. 
 We could then construct $u^1$, $u^2$ using Green's identities \eqref{eq:Green}. Actually, for the exterior solution it holds as well
\begin{equation}\label{eq:Green2}
 u^1=- SL_1(\gamma_N^1 u^t)+ DL_1(\gamma_D u^t)
\end{equation}
which is consequence of the definition of the  total wave $u^t$ and of the identity \eqref{eq:Cinc2} which ensures that the contribution from $u^{inc}$ in the representation formula cancels out.

 We first have the  following first kind integral equations due to \cite{costabel-stephan} with a positive definite  principal part 
%
%
\begin{equation}\label{eq:system_trans_FK}
{\rm CFK} \begin{bmatrix}
 \gamma_D  {u^t}\\
  \gamma^1_N  {u^t}
      \end{bmatrix}:=
\begin{bmatrix}
  -(K_1+K_2) & ({\rho}^{-1} S_2 + S_1) \\
  -(N_1 +{\rho} N_2)&(K^\top_1+K^\top_2)
  \end{bmatrix}\begin{bmatrix}
  \gamma_D{u^t}\\
  \gamma^1_N{u^t}
      \end{bmatrix}= \begin{bmatrix}  
   \gamma_D  u^{inc}\\
  \gamma_N u^{inc}
    \end{bmatrix} .
\end{equation}
This system of integral equations can be easily derived from \eqref{eq:C1}-\eqref{eq:Cinc} and the boundary conditions \eqref{eq:bc}.
Obviously, these integral equations are understood in the a.e. sense.
We refer to the system of boundary integral equations~\eqref{eq:system_trans_FK} as CFIEFK for {\em combined field integral equation of first kind}.


A widely used first kind boundary integral formulation of the transmission problem~\eqref{eq:Ac_i}-\eqref{eq:bc} is due to Kress and Roach cf. \cite{KressRoach} and it consists of the following pair of integral equations:
\begin{equation}\label{eq:system_trans}
{\rm CSK}\begin{bmatrix}
 \gamma_D{u^t}\\
  \gamma^1_N{u^t}
      \end{bmatrix}:=
      \Bigg(\frac{{\rho}^{-1}+1}{2}\begin{bmatrix}
                                 I&\\
                                 &I
                                \end{bmatrix}
                                +\begin{bmatrix}
                                   (K_2-{\rho}^{-1}K_1)&{\rho}^{-1} (S_1 - S_2)\\
                                   -(N_1 - N_2)&(K_1^\top-{\rho}^{-1} K^\top_2)
                                 \end{bmatrix}
\Bigg)\begin{bmatrix}
 \gamma_D{u^t}\\
  \gamma^1_N{u^t}
      \end{bmatrix}= \begin{bmatrix}  
      {\rho}^{-1}\gamma_D  u^{inc}\\
  \gamma_N u^{inc}
    \end{bmatrix}.
\end{equation}
In what follows we refer to the integral equations~\eqref{eq:system_trans} as CFIESK ({\em combined field integral equation of second kind}).

We introduced recently Generalized Combined Source Integral Equation (GCSIE) formulations of transmission problems~\cite{turc2,turc3} and we established their well-posedness in the case when $\Gamma$ is regular enough. We consider here a {\em direct} counterpart of the GCSIE formulations which can be obtained by combining the previous formulations in the form
\begin{subequations}
\label{eq:GFK}
\begin{equation}
\label{eq:GFK1}
\bigg(\frac{\rho}{\rho+1} {\rm CSK}+\frac{2}{1+\rho}\begin{bmatrix}
      &S_{\kappa}\\
      -\rho N_{\kappa}
      \end{bmatrix}{\rm CFK}\bigg)\begin{bmatrix}
 \gamma_D{u^t}\\
  \gamma^1_N{u^t}
      \end{bmatrix}=\frac{1}{\rho+1}\begin{bmatrix}
      I&2S_{\kappa}\\
      -2\rho N_\kappa&\rho I
      \end{bmatrix} \begin{bmatrix}  
   \gamma_D  u^{inc}\\
  \gamma_N u^{inc}
    \end{bmatrix}.
\end{equation} 
Here, $\kappa$ is a complex wave number with positive imaginary part. This parameter can be appropriately taken to improve the spectrum of the underlying operator. This leads  to faster convergence, when discretized, of Kyrlov iterative methods, such as GMRES, for the linear system 
(see subsection 5.4). 
Defining
\[
 {\rm R}:=\frac{1}{\rho+1}\begin{bmatrix}
      I&2S_{\kappa}\\
      -2\rho N_\kappa&\rho I
      \end{bmatrix} 
\]
we observe that \eqref{eq:GFK1} can be  written as 
\begin{equation}
\label{eq:GFK2} {\rm GFK} \begin{bmatrix} 
 \gamma_D{u^t}\\
  \gamma^1_N{u^t}
      \end{bmatrix} :=
\begin{bmatrix}
 \frac{1}2I+K_2&-\rho^{-1}S_2\\
\rho N_2& \frac{1}2I
-K_2^\top&
\end{bmatrix}\begin{bmatrix}
 \gamma_D{u^t}\\
  \gamma^1_N{u^t}
      \end{bmatrix}+{\rm R}\:{\rm CFK}\begin{bmatrix}
 \gamma_D{u^t}\\
  \gamma^1_N{u^t}
      \end{bmatrix}= {\rm R}     \begin{bmatrix}  
   \gamma_D  u^{inc}\\
  \gamma_N u^{inc}
    \end{bmatrix}.
\end{equation}
\end{subequations}
We will refer to this formulation as  {\em Regularized Combined Field Integral Equations} (CFIER).

Another possible formulation of the transmission problem~\eqref{eq:Ac_i}-\eqref{eq:bc} takes on the form of {\em single} integral equations~\cite{KleinmanMartin}. 
The main idea is to look for the field $u^2$ as a single layer potential, that is
\[
u^2(\mathbf{z})=-2[SL_2\mu](\mathbf{z}),\ \mathbf{z}\in D_2,
\] 
where $\mu$ is an {\em unphysical} density defined on $\Gamma$.  The transmission boundary conditions~\eqref{eq:bc} and the trace relations \eqref{traces} imply first
\[
 \gamma_D   u^t=-2S_2\mu,\quad 
 \gamma_N^1  u^t=-\rho(I+2K_2^\top)\mu, 
\]
and second, from \eqref{eq:Green2}, 
\[
u^1(\mathbf{z})=\rho SL_1[(I+2K_2^\top)\mu](\mathbf{z})-2DL_1[S_2\mu](\mathbf{z}),\quad \mathbf{z}\in D_1.
\]  
Using these representations of the fields $u^1$ and $u^2$, the Neumann and Dirichlet traces of $u^1$ and $u^2$ on $\Gamma$ are used in a Burton-Miller type combination of the form  
$(\gamma_N^1 u^1-i\eta  {\gamma_D^1} u^1)-(\rho\gamma_N^2 u^2-i\eta {\gamma_D^2} u^2)= {-\gamma_Nu^{inc}+i\eta \gamma_D u^{inc}}$ cf~\cite{BurtonMiller} to lead to the following boundary integral equation  
\begin{equation}\label{eq:single}
{\rm SIE}\:\mu:=-\frac{1+\rho}{2}\mu+\mathbf{K}\mu-i\eta\mathbf{S}\mu=
 {\gamma_N u^{inc}}-i\eta  {\gamma_D u^{inc}}
,\quad 0\ne \eta\in\mathbb{R}
\end{equation}
where
$$\mathbf{K}=-K_2^\top(\rho I-2K_2^\top)-\rho K_1^\top(I+2K_2^\top)+2(N_1-N_2)S_2$$
and
$$\mathbf{S}=-\rho S_1(I+2K_2^\top)-(I-2K_1)S_2,$$

We refer in what follows to equation~\eqref{eq:single} as SCFIE. The coupling parameter $\eta$ in equations~\eqref{eq:single} is typically taken to be equal to $k_1$.

\section{Existence and uniqueness for the boundary integral formulations CFIEFK~\eqref{eq:system_trans_FK}, CFIESK~\eqref{eq:system_trans}, CFIER~\eqref{eq:GFK}, and SCFIE~\eqref{eq:single}}\label{well-posedness}

In this section we state the well-posedness of the boundary integral formulations {on Lipschitz curves} presented in previous section. {Let us point out that, with very minor and direct modifications, the proofs of the main results of this section can be adapted to the case of Lipschitz domains in $3D$. Hence, these formulations are well-posed as well, and the  associated integral operators satisfy the same properties in the Sobolev frame. Since in this paper we have focused our attention to bidimensional domains, we have preferred to work only on Lipschitz curves for the sake of simplicity.

One of the basic tools we will use in the analysis developed in this section consists in comparing the boundary operators for the Helmholtz equation with those for Laplace operator, $S_0, K_0, K_0^\top,$ and $N_0$  which are defined as the corresponding boundary Helmholtz operators in section 3 using 
\[
G_0({\bf x})=-\frac{1}{2\pi}\log|{\bf x}|  ,
\]
the fundamental solution of the Laplace operator, instead. 
}

\subsection{Sobolev spaces and functional properties of boundary operators}
For any $D\subset\mathbb{R}^2$ domain with bounded Lipschitz boundary $\Gamma$,  we denote by $H^s(D)$ the classical Sobolev space of order $s$ on $D$ 
(see for example in the excellent textbooks \cite[Ch. 3]{mclean:2000} or \cite[Ch. 2]{adams:2003}). 
We consider in addition the Sobolev spaces defined on the boundary $\Gamma$,  $H^s(\Gamma)$, which are well defined for any $s\in[-1,1]$. It is well known that for  smoother  $\Gamma$ the range for these $s$ can be widened but we do not make use of these spaces in this section. We recall that for any $s>t$, $H^s(\Sigma)\subset H^t(\Sigma)$, $\Sigma\in\{D_1,D_2,\Gamma\}$ with  continuous and compact embedding. Moreover, and
$\big(H^t(\Gamma)\big)'=H^{-t}(\Gamma)$ when the inner product of $H^0(\Gamma)=L^2(\Gamma)$ is used as duality product.

It is well known that $\gamma_D^j: H^{s+1/2}(D_j)\to H^s(\Gamma)$ is continuous for $s\in(0,1)$ and if
\[
 H^s_\Delta(D_j):=\left\{U\in H^{s}(D_j)\ :\ \Delta U\in L^2(D_j) \right\},
\]
endowed with its natural norm, 
then $\gamma_N:  H^s_\Delta(D_j)\to H^{s-3/2}(\Gamma)$ is continuous  for  $s\in(1/2,3/2)$.

The space $H^1(\Gamma)$, and its dual $H^{-1}(\Gamma)$, are then the limit case from several different perspectives. Let $\nabla_\Gamma$ be the tangential derivative defined as 
\[
 \nabla_\Gamma (U|_\Gamma) =                           
\partial_s (U|_\Gamma)\, \bm{\tau}\quad
\]
where $\bm{\tau}$ denotes (one of) the unit tangent vector fields to $\Gamma$, and $\partial_s$ the tangential derivative with respect to ${\bm \tau}$. It is known  that an integral expression for the (or an equivalent) inner product {in $H^1(\Gamma)$} is given by
\begin{eqnarray}\label{eq:normH1}
(u,v)_{H^1(\Gamma)}:=\int_{\Gamma} \nabla_\Gamma u \cdot \overline{\nabla_\Gamma v} +\int_{\Gamma} u\overline{v}
&=&
\langle \bm{\Lambda}u, \overline{\bm{\Lambda}v}\rangle,\quad \bm{\Lambda}u:=\nabla_\Gamma u+ u \,{\bf n}
\end{eqnarray}
We commit here a slight abuse of notation, and denote with the same symbol  $\langle \cdot,\cdot \rangle$ the non-complex integral inner product in $\big(L^2(\Gamma)\big)^2$.

\begin{theorem}
Let  $D_2$ be a bounded domain, with Lipschitz boundary $\Gamma$, and set $D_1:=\mathbb{R}^2\setminus \overline{D}_2$. Then, if $\chi\in{\cal C}^\infty( {\mathbb{R}^2})$ with compact support, 
\begin{itemize}
\item $SL_k:H^{s}(\Gamma)\to H^{s+3/2}(D_2)$, $\chi SL_k:H^{s}(\Gamma)\to H^{s+3/2}(D_1)$
\item $DL_k:H^{s+1}(\Gamma)\to H^{s+3/2}(D_2)$, $\chi DL_k:H^{s+1}(\Gamma)\to H^{s+3/2}(D_1)$
\end{itemize}
are continuous for $s\in[-1,0]$. 
Moreover 
\begin{itemize}
\item $S_k:H^{s}(\Gamma)\to H^{s+1}(\Gamma)$
\item $K_k:H^{s+1}(\Gamma)\to H^{s+1}(\Gamma)$
\item $K^\top_k:H^{s}(\Gamma)\to H^{s}(\Gamma)$
\item $N_k:H^{s+1}(\Gamma)\to H^{s}(\Gamma)$
\end{itemize}
are continuous for $s\in[-1,0]$.
\end{theorem}
A proof for the intermediate values $s\in(-1,0)$ can be found in \cite[Th 1]{MR937473} (see also \cite[Th 6.12]{mclean:2000}). 
The proof for 
$k=0$, the Laplace operator, and $s=-1,0$ can be found in \cite{Ver:1984} (see also the comments following Th 6.12 in \cite{mclean:2000}). For $k\ne 0$ the argument follows by showing that the difference between the corresponding Laplace and Helmholtz boundary  operators are smooth enough.  We refer to the discussion  at the end of Chapter 6 in \cite{mclean:2000} or that following Theorem 1 in
\cite{MR937473} and references therein. For the sake of completeness, we will give next a proof   of this result. 
This result will be actually used to prove the well posedness of the different formulations considered in this paper.  

We will need first this technical lemma which will be proven for the sake of completeness.

\begin{lemma}\label{lemma:kernels}
The integral operators 
\[
 \bm{\Lambda} (K_{k_1}-K_{k_2}),\quad  N_{k_1}-N_{k_2}
\]
have  weakly (integrable) singular kernels. 
\end{lemma}
\begin{proof}
Without loss of generality we can assume in this proof that  $k_2=0$.  
The kernel of $ \bm{\Lambda} (K_{k_1}-K_{0})$ is then given (a.e.) by
\[
-{\bf n}({\bf x})\left[ \left(\nabla(G_{k_1}-G_0)( {\bf x}-{\bf y})\right)\cdot {\bf n}({\bf y})\right]
-{\bm\tau}({\bf x})\left[{\bm \tau}^{\top}({\bf x})\left(\nabla^2(G_{k_1}-G_0)({\bf x}-{\bf y})\right) {\bf n}({\bf y})\right],
\quad {\bf x},{\bf y}\in\Gamma. 
\]
( $\nabla^2 G$ above denotes the Hessian matrix of $G$),
whereas  
\[
 -{\bf n}^{\top}({\bf x})\left(\nabla^2(G_{k_1}-G_0)({\bf x}-{\bf y})\right) {\bf n}({\bf y})
\]
is  the kernel of $N_{k_1}-N_{0}$.

Then it suffices to show that for any $R>0$ there exists
$C_R>0$ such that 
\begin{equation}\label{eq:new01:lemma42}
 |\nabla (G_{k_1}-G_0)({\bf x})|+\|\nabla^2(G_{k_1}-G_0)({\bf x})\| \le C_R(1+\log|{\bf x}|),\quad  \text{for } 0<|{\bf x}|<R.
\end{equation}
where $\|\cdot\|$ above is any matrix norm.  

The proof of this  bound relies on analytical properties of the Hankel functions, namely, the behavior at zero and appropriate decompositions of these functions. We point out that a deeper analysis on this topic will be carried out in Section~\ref{singular_int}, so that we limit ourselves the exposition of the 
properties we are going to use. 

With the identity $( H_0^{(1)})'(z)=- {H_1^{(1)}}(z)$ we can obtain first
\begin{equation}\label{eq:new02:lemma42}
 \nabla (G_{k_1}-G_0)({\bf x})=\left(-\frac{i}{4}  H_1^{(1)}(k_1|{\bf x}|){k_1}|{\bf x}|+\frac{1}{2\pi}\right)
 \frac{{\bf x}^\top}{|{\bf x}|^2}
\end{equation}
(we follow the convention of writing the gradient as a row vector). Note now that
\[
\frac{i}4 z {H^{(1)}_1(z)} =-\frac{1}{2\pi}J_1(z) z \log z+\frac{1}{2\pi}+d_1(z)z^2
\]
where $J_1$ is the Bessel function which is known to be smooth,  with $J_1(0)=0$, $J_1'(0)=1/2$  and 
 {$d_1$} being smooth as well. From this decomposition we deduce easily that $\nabla (G_{k_1}-G_0)$ satisfies \eqref{eq:new01:lemma42}.

On the other hand, using that
\[
 {\left( zH_1^{(1)}(z)\right)'= z H_0^{(1)}(z)}. 
\]
we can easily show that 
\begin{eqnarray*}
 \nabla^2(G_{k_1}-G_0)({\bf x})&=& -\frac{i{k_1^2}}{4}H_0^{(1)}(k_1|{\bf x}|)\frac{1}{|{\bf x}|^2}
{\bf x} \: {\bf x} ^{\top}+\left(\frac{i}{4} H_1^{(1)}(k_1|{\bf x}|)k_1|{\bf x}|-\frac1{2\pi}\right)
 \left(\frac{2}{|{\bf x}|^4}{\bf x}\:  {\bf x} ^{\top} - \frac{1}{|{\bf x}|^2}I\right)
\end{eqnarray*}
where $I$ is $2\times 2$ identity matrix. 

 Since in addition
\[
\frac{i}4   H_0^{(1)}(z) =-\frac{1}{2\pi}J_0(z)  \log z+c_0(z)
\]
where $J_0$ (the Bessel function) and $c_0$ are  smooth,
a simple inspection shows that the entries of the Hessian matrix  can be bounded by $C \log(|{\bf x}|)$ on any punctured ball around ${\bf 0}$ with appropriate constant $C$ (which depends obviously on ${k_1}$ and on the diameter of the domain).
\end{proof}

\begin{theorem}\label{theo:regularity} Let $k_1\ne k_2$. Then  
\begin{itemize}
\item $S_{k_1}-S_{k_2}:H^{-1}(\Gamma)\to H^{1}(\Gamma)$
\item $K_{k_1}-K_{k_2}:H^{0}(\Gamma)\to H^{1}(\Gamma)$
\item $K^\top_{k_1}-K^\top_{k_2}:H^{-1}(\Gamma)\to H^{0}(\Gamma)$
\item $N_{k_1}-N_{k_2}:H^{0}(\Gamma)\to H^{0}(\Gamma)$.
\end{itemize}
are continuous and compact. 
\end{theorem}
\begin{proof} 
It is a well known consequence of Lax Theorem (see \cite[Th 4.12]{Kress}) that  integral operators with weakly singular kernels define   compact operators in $L^2$. From this fact, and Lemma \ref{lemma:kernels}, we see that 
\[
N_{k_1}-{N_{k_2}}:H^0(\Gamma)\to H^0(\Gamma) 
\]
is compact. 

Similarly, 
\[
\bm{\Lambda}(K_{k_1}-K_{k_2}): H^0(\Gamma)\to \big[H^0(\Gamma)\big]^2 
\]
is also compact (again from Lemma \ref{lemma:kernels}). 

Consider now a  weakly  convergent sequence $(u_n)_n$. Then $\big(\bm{\Lambda}(K_{k_1}-K_{k_2}){u_n}\big)_n$ converges strongly, by the compactness of the operator, in $\big[H^0(\Gamma)\big]^2$  and, because \eqref{eq:normH1}, $\big((K_{k_1}-K_{k_2}){u_n}\big)_n$ is strongly convergent in $H^1(\Gamma)$. In other words, $K_{k_1}-K_{k_2}:H^0(\Gamma)\to H^1(\Gamma)$ is compact as well. 

A duality argument proves now the result for $K_{k_1}^\top-K_{k_2}^\top$.

Finally, assume without loss of generality, that $k_1$ is taken so that $N_{k_1}:H^{0}(\Gamma)\to H^{-1}(\Gamma)$ is  invertible. (It suffices to take as
$k_1$   a pure imaginary number). Then
\begin{eqnarray}
 S_{k_1}-S_{k_2}&=&N_{k_1}^{-1}\big[N_{k_1} S_{k_1}-N_{k_2} S_{k_2}\big]+N_{k_1}^{-1}\big(N_{k_2}-N_{k_1})S_{k_2}\nonumber\\
 &=& N_{k_1}^{-1}\big[\big(K_{k_1}^\top\big)^2- \big(K_{k_2}^\top\big)^2\big]+N_{k_1}^{-1}\big(N_{k_2}-N_{k_1})S_{k_2}\nonumber\\
 &=&N_{k_1}^{-1}\big[K^\top_{k_1}-K^\top_{k_2}\big]K_{k_2}^\top+
 N_{k_1}^{-1}K_{k_1}^\top\big[K^\top_{k_1}-K^\top_{k_2}\big]+N_{k_1}^{-1}\big(N_{k_2}-N_{k_1})S_{k_2} \label{eq:Sk1-Sk2}
\end{eqnarray}
where we have applied  the second identity in \eqref{eq:calderon}.
In view of \eqref{eq:Sk1-Sk2} and the mapping properties of the operators involved, we conclude that $ S_{k_2}-S_{k_1}:H^{-1}(\Gamma)\to H^1(\Gamma)$ is continuous and compact.  The proof is now finished.
\end{proof}

The last ingredient in our proof is  this result due to Escauriaza, Fabes and Verchota~\cite{EsFaVer:1992}. In this result, $K_0$, $K_0^\top$ are the double and adjoint double layer operator for Laplace equation (which obviously correspond to $k=0$).

\begin{theorem}\label{theo:inv}
For any Lipschitz {curve} and  $\lambda\not\in {(-1/2,1/2]}$, the mappings
\[
 \lambda I+K_0 :H^s(\Gamma)\to H^s(\Gamma),\quad
 \lambda I+K_0^\top :H^{-s}(\Gamma)\to H^{-s}(\Gamma)
\]
are invertible for $s\in[0,1]$. Moreover
\[
 \frac12I+K_0: H^s(\Gamma)\to H^s(\Gamma),\quad
 \frac12I+K^\top_0: H^{-s}(\Gamma)\to H^{-s}(\Gamma)
\]
are Fredholm of index 0.  
\end{theorem}  
\begin{proof}
In  \cite[Th 3.1 and 3.3]{Ver:1984} it is proven that $-\tfrac12 I+K_0:H^s(\Gamma)\to H^s(\Gamma)$ is invertible for $s=0,1$. (Note that in this paper $\frac{1}{2\pi}\log|{\bf x}|=-G_0({\bf x})$ is taken as fundamental solution.). By interpolation of Sobolev spaces and a transposition argument, we can prove the result for $\lambda=-1/2$. 
 
 In  \cite[Th 2]{EsFaVer:1992} the result is proven for $|\lambda|>1/2$ and $s=0$. To extend the result for the remaining $s$ we can adapt the argument outlined in \cite{Ver:1984}. Suppose that $S_0:H^0(\Gamma)\to H^1(\Gamma)$ is invertible. Then, $S_0(\lambda I+K_0)S_0^{-1}:H^1(\Gamma)\to H^1(\Gamma)$ is invertible as well. But Calderon identities \eqref{eq:calderon}
imply
\begin{equation}
S_0(\lambda I+K_0)S_0^{-1}=(\lambda I+K_0^\top)S_0S_0^{-1}=(\lambda I+K_0^\top)\label{eq:01:theo:inv}
\end{equation}
from where we conclude the invertibility, for   $\lambda\not\in (-1/2,1/2]$, of $\lambda I+K_0^\top:H^s(\Gamma)\to H^{s}(\Gamma)$ first  for $s=-1$ and next, by interpolation, for $s\in [-1,0]$. Again, by transposition, we can show that $\lambda I+ {K_0}:H^s(\Gamma)\to H^{s}(\Gamma)$ is invertible for $s\in[0,1]$. This argument breaks down when $S_0$ is not invertible, that is, when the logarithmic capacity of the curve $\Gamma$ is $1$. For this case we consider
\[
 \tilde{S}_0\varphi:=S_0 {\varphi}+\int_{\Gamma }\varphi.
\]
It can be shown that $\tilde{S}_0:H^s(\Gamma)\to H^{s+1}(\Gamma)$ is continuous and invertible for $s\in[-1,0]$ (see \cite[Ch. 8]{mclean:2000}). Moreover, since  
\[
 \int_\Gamma K_0\varphi =K_0^\top \bigg(\int_\Gamma \varphi\bigg)
\]
we still have
 $  K_0\tilde{S}_0=\tilde{S}_0 K^\top_0$, which makes possible to extend the 
 argument in \eqref{eq:01:theo:inv}, with $ \tilde{S}_0$ instead, to this case as well.

Finally, in \cite[Th 4.2]{Ver:1984} it is proven that $\frac{1}2I+K_0^\top:L^2_0(\Gamma)\to L^2_0(\Gamma)$ is invertible where $L^2_0(\Gamma)$ is the subspace of functions in  $L^2(\Gamma)=H^0(\Gamma)$ which zero mean. Then $\dim N(\frac{1}2I+K_0^\top)=
\mathop{\rm codim}\: R(\frac{1}2I+K_0^\top)=1$, which in particular implies that $\frac{1}2I+K_0^\top:H^0(\Gamma)\to H^0(\Gamma)$ is Fredholm of index zero and, by transposition, so is $\frac{1}2I+K_0:H^0(\Gamma)\to H^0(\Gamma)$. The same argument as before extends this result to $H^s(\Gamma)$. 
\end{proof}

Observe that in the proof of this last result we have assumed that $\Gamma$ is simply connected. Otherwise $\dim N(\frac{1}2I+K_0^\top)=
\mathop{\rm codim}\: R(\frac{1}2I+K_0^\top)$ equals to the number of simply connected components, and the argument is still valid with this very minor modification.

\subsection{Well-posednees of the boundary integral equation formulations CFIEFK~\eqref{eq:system_trans_FK}, CFIER~\eqref{eq:GFK}, and SCFIE~\eqref{eq:single}}

In this section we state and prove the well-posedness of the integral equations formulations for the transmission problems. For  Costabel-Stephan CFIEFK \eqref{eq:system_trans_FK}   and Kress-Roach CFIESK \eqref{eq:system_trans} formulations, the stability has has been already proved, but only in the space $H^{1/2}(\Gamma)\times H^{-1/2}(\Gamma)$~\cite{costabel-stephan,ToWe:1993}. We will however extend these results to a wider range of Sobolev spaces. To this end, we make use of the skew-symmetric  bilinear form 
\[
\langle (f,\varphi),(g,\psi) \rangle:= \int_\Gamma f\psi-\int_\Gamma g\varphi,\quad (f,\varphi),(g,\psi)\in H^{1/2}(\Gamma)\times H^{-1/2}(\Gamma)
\] 
which gives a non-usual representation, but more convenient for our purposes,  of the duality product
between 
$H^{1/2}(\Gamma)\times H^{-1/2}(\Gamma)$ and itself. Obviously, the integrals above have to be understood in a weak sense if $\varphi$ and $\psi$ are not sufficiently smooth.   

We point out that if ${\rm A}:H^{1/2}(\Gamma)\times H^{-1/2}(\Gamma)\to H^{1/2}(\Gamma)\times H^{-1/2}(\Gamma)$, with
${\rm A}=(A_{ij})_{i,j=1}^2$, then the transpose operator 
\[
 \langle {\rm A}^t(f,\varphi),(g,\psi) \rangle:=
 \langle (f,\varphi),{\rm A}(g,\psi) \rangle,\quad \forall (f,\varphi),(g,\psi)\in H^{1/2}(\Gamma)\times H^{-1/2}(\Gamma)
\]
is given by
\begin{equation}\label{eq:at}
 {\rm A}^t=\begin{bmatrix}
      &1\\
      -1
     \end{bmatrix}
     {\rm A}^\top
\begin{bmatrix}
      &1\\
      -1
     \end{bmatrix}
=\begin{bmatrix}
      A_{22}^\top&-A_{12}^\top \\[1.2ex]
      -A_{21}^\top&A_{11}^\top
     \end{bmatrix}. 
\end{equation}
Here, $A_{ij}^\top$ is simply the adjoint of $A_{ij}$ in the bilinear form defined by the integral product in $\Gamma$.  {Observe that the familiar identity $(AB)^t=B^t A^t$ still holds.

\begin{theorem}~\label{well-posedness-CFIEFK}
The following operators 
 \[  
  {\rm CSK}:H^{s}(\Gamma)\times H^{s-1}(\Gamma)\to
  H^{s}(\Gamma)\times H^{s-1}(\Gamma),\quad
  {\rm CFK}:H^{s}(\Gamma)\times H^{s-1}(\Gamma)\to
  H^{s}(\Gamma)\times H^{s-1}(\Gamma)
 \] 
 are invertible with continuous inverses for all $s\in[0,1]$ 
 \end{theorem}
\begin{proof} 
Assume $\rho\ne 1$. Note that from  \eqref{eq:system_trans}
\begin{eqnarray*}
 {\rm CSK}&=&\frac{\rho-1}{\rho}\underbrace{\begin{bmatrix}
\frac{\rho+1}{2(\rho-1)}I+K_0&\\& {\frac{\rho+1}{2(\rho-1)}} I+K^\top_0 
                                               \end{bmatrix}}_{\rm CSK_0}\\
                                               &&+
                                               \underbrace{\begin{bmatrix}
                                                             K_2-K_0-\rho^{-1}(K_1-K_0) & \rho^{-1}(S_1-S_2)\\
                                                            -(N_1-N_2)& K^\top_1-K^\top_0-\rho^{-1}(K^\top_2-K^\top_0 )
                                                           \end{bmatrix}
}_{
 {\rm L}_1}
\end{eqnarray*}
($K_0,K_0^\top$ denote here the double layer and adjoint double layer operator for the Laplace equation). 
Clearly,
\[
\frac{\rho+1}{2(\rho-1)}\in\mathbb{R}\setminus[-1/2,1/2]. 
\]
By Theorem \ref{theo:regularity}, 
${\rm L}_1:H^{0}(\Gamma)\times H^{-1}(\Gamma)\to H^{1}(\Gamma)\times H^{0}(\Gamma)$ is continuous and compact. On the other hand,  ${\rm CSK_0}:H^{s}(\Gamma)\times H^{s-1}(\Gamma)\to
H^{s}(\Gamma)\times H^{s-1}(\Gamma)$ is invertible
 by Theorem \ref{theo:inv}. Therefore, ${\rm CSK}$ is Fredholm  of index zero  with kernel in $H^1(\Gamma)\times H^0(\Gamma)$. For $\rho=1$ the proof is even simpler since Theorem \ref{theo:regularity} shows now that 
$ {\rm CSK}$ is a compact perturbation of the identity in $H^{s}(\Gamma)\times H^{s-1}(\Gamma)$ and the same argument can be applied to prove that the kernel is contained in $H^1(\Gamma)\times H^0(\Gamma)$ as well.

The injectivity can be shown using Green identities and the uniqueness of the transmission problem cf. \cite{costabel-stephan}. We emphasize that the regularity of the kernel of {\rm CSK} is crucial to make the arguments in \cite{costabel-stephan} valid.

For the second operator, we first note that
\[
 {\rm CFK}= \underbrace{\begin{bmatrix} 
             -2 K_0& \frac{\rho+1}{\rho} S_0\\
             -(1+\rho) N_0&2K_0^\top 
            \end{bmatrix}}_{{\rm CFK}_0}+{\rm L}_2
\]
where ${\rm L}_2:H^{0}(\Gamma)\times H^{-1}(\Gamma)\to H^{1}(\Gamma)\times H^{0}(\Gamma)$ is again compact.  Calder\'{o}n identities in \eqref{eq:calderon} yield that (see also \eqref{eq:at})
\begin{eqnarray*}
  {\rm CFK}_0^t\:  {\rm CFK}_0&=&\begin{bmatrix}
              2 K_0&- \frac{\rho+1}{\rho} S_0\\
              (1+\rho) N_0&-2K_0^\top                                    
                                  \end{bmatrix}
              \begin{bmatrix} 
             -2 K_0& \frac{\rho+1}{\rho} S_0\\
             -(1+\rho) N_0&2K_0^\top 
            \end{bmatrix}\\
            &=&
            \begin{bmatrix}
               -4 K^2_0+\frac{(\rho+1)^2}{\rho} S_0N_0&\frac{2(\rho+1)}{\rho}(K_0S_0-S_0K_0^\top)\\
                       -{2(\rho+1)}(N_0K_0-K_0^\top N_0)&-4 \big(K^\top_0\big)^2+\frac{(\rho+1)^2}{\rho} N_0S_0\\
            \end{bmatrix}\\
            &=&
            \frac{(\rho-1)^2}{\rho}
            \begin{bmatrix}
               -\frac{(\rho+1)^2}{4(\rho-1)^2}I+ K^2_0& \\
                       &   -\frac{(\rho+1)^2}{4(\rho-1)^2}I+ (K^\top_0)^2\\
            \end{bmatrix}\\
            &=&
            \frac{(\rho-1)^2}{\rho}
            \begin{bmatrix}
               (-\beta I+K_0)(\beta I +K_0)& \\
                       &   
               (-\beta I+K^\top_0)(\beta I +K^\top_0)
            \end{bmatrix},\quad \beta:=\frac{\rho+1}{2\rho-2}.  
\end{eqnarray*}
Theorem \ref{theo:inv}, observe that  {$|\beta|> 1/2$}, implies that this  mapping is invertible, and therefore, so is $ {\rm CFK}_0$.
We then conclude that ${\rm CFK}:H^{s}(\Gamma)\times H^{s-1}(\Gamma)\to H^{s}(\Gamma)\times H^{s-1}(\Gamma)$ is Fredholm   of index zero . 
 If $\rho=1$,
\[
 {\rm CFK}_0^t\:  {\rm CFK}_0= {-}\begin{bmatrix}
                               I&\\
                               &I
                              \end{bmatrix}
\]
which makes the argument still valid. The injectivity, which implies the invertibility by Fredholm alternative, follows using classical arguments  cf. \cite{ToWe:1993}.
%
%
%
\end{proof}

We would like to point out that the original proof for the invertibility of {\rm CFK} is based, in a nutshell, on showing that its principal part is coercive (strongly elliptic) in the space $H^{1/2}(\Gamma)\times H^{-1/2}(\Gamma)$   cf.  \cite[Section 5]{costabel-stephan}.

\begin{theorem} Assume that the wavenumber ${\kappa}$ in  \eqref{eq:GFK}  has a positive imaginary part. Then
the operators
 \[
 {\rm GFK}:H^{s}(\Gamma)\times H^{s-1}(\Gamma)\to
  H^{s}(\Gamma)\times H^{s-1}(\Gamma)
 \]
are invertible with continuous inverses in the spaces $H^{s}(\Gamma)\times H^{s-1}(\Gamma)$ for all $s\in[0,1]$. 
 \end{theorem}
\begin{proof}
We first prove that ${\rm GFK}$ is a compact perturbation of an invertible operator.  Starting from 
\eqref{eq:GFK2}, and proceeding as above, we  derive  
\[ 
{\rm GFK}=
\underbrace{\begin{bmatrix}
 \frac{1}2I+K_0&-\rho^{-1}S_0\\
\rho N_0& \frac{1}2I-K_0^\top&
\end{bmatrix}}_{L_0}
 +\underbrace{\frac{1}{\rho+1}\begin{bmatrix}
      I&2S_{0}\\
      -2\rho N_0&\rho I
      \end{bmatrix}}_{\rm R_0}  {\rm CFK}_0   +{\rm L}_3=:{\rm GFK}_0+{\rm L}_3.
      \]
with  ${\rm L}_3:H^{0}(\Gamma)\times H^{-1}(\Gamma)\to H^{1}(\Gamma)\times H^{0}(\Gamma)$ compact. Straightforward calculations  show that
\begin{eqnarray*}
 L_0^t L_0&=&0\\
 {\rm CFK}_0^t\:{\rm R}_0^t \:{\rm R_0} \: {\rm CFK}_0&=&\\[1.5ex]
 && \hspace{-4.5cm}  \frac{4(\rho-1)^2}{(\rho+1)^2}
 \begin{bmatrix}
 K_0^2\big(-\frac{\rho+1}{2(\rho-1)} I+K_0\big) 
 \big(\frac{\rho+1}{2(\rho-1)} I+K_0\big)\\         
 &
 \big(K_0^\top\big)^2
 \big(-\frac{\rho+1}{2(\rho-1)} I+K^\top_0\big) 
 \big(\frac{\rho+1}{2(\rho-1)} I+K^\top_0\big)\\    
                                                  \end{bmatrix}\\[1.5ex]
L_0^t\:  {\rm R}_0 \:{\rm CFK}_0+ {\rm CFK}_0^t\: {\rm R}_0^t \:{\rm L}_0  &=&\\
&& \hspace{-4.5cm}   -\frac{2({\rho-1})}{\rho+1}\begin{bmatrix}
     \left(-I+2 K_0^2\right) 
     (\frac{\rho +1}{2(\rho -1)} I+K_0)  &   \\ & \left(-I+2(K_0^\top)^2\right) \left(\frac{\rho +1}{2  (\rho -1)}I+(K_0^\top)\right).
   \end{bmatrix}
                  \end{eqnarray*}
Therefore, 
\begin{equation}\label{eq:GFK0}
  {\rm GFK}_0^t\:{\rm GFK}_0=\begin{bmatrix} A_1&\\
 &A_1^\top \end{bmatrix}
\end{equation}
where
\[
 A_1:= {\frac{4(\rho-1)^2}{(\rho+1)^2}}\Big(K_0+\frac{\rho+1}{2(\rho-1)} I\Big)\Big(K_0^3-\frac{3 (\rho +1)}{ 2(\rho-1)}K_0^2+\frac{\rho +1}{2(\rho-1)} I\Big)
\]
Note that
\[
 A_1= {\frac{4(\rho-1)^2}{(\rho+1)^2}}\Big(\frac{\rho+1}{2(\rho-1)} I+K_0\Big) \big(-\lambda_1(\rho)I+K_0\big)\big(-\lambda_2(\rho)I+K_0\big)\big(-\lambda_3(\rho)I+K_0\big)
\]
where $\lambda_j(\rho)$ are the roots of 
\[
p_\rho(x)=x^3-\frac{3 (\rho +1)}{{2(\rho-1)}}x^2+\frac{\rho +1}{{2(\rho-1)} }
\]
which can be  shown to be all real, and lying outside of $[-1/2,1/2]$ for any $\rho\in (0,\infty)\setminus\{1\}$ cf. Lemma \ref{lemma:polynomial} below. Theorem \ref{theo:inv} implies the invertibility of $A_1$ and therefore of $ {\rm GFK}_0$. 

For $\rho=1$  the proof clearly breaks down, but it can be studied separately and deduce that \eqref{eq:GFK0} still holds with
\[
 A_1=1-3K_0^2=-3\Big(\frac{1}{\sqrt{3}}I+K_0\Big)\Big(-\frac{1}{\sqrt{3}}I+K_0 \Big)
\]
which again is invertible. 

In short, we have shown that $ {\rm GFK}:H^s(\Gamma)\times H^{s-1}(\Gamma)\to H^s(\Gamma)\times H^{s-1}(\Gamma)$, for $s\in[0,1]$, is Fredholm of index zero.  To prove it is invertible we will consider the adjoint operator given by (see \eqref{eq:GFK2})
\[
 {\rm GFK}^t= 
\begin{bmatrix}
 \frac{1}2I-K_2&\rho^{-1}S_2\\[1.2ex]
-\rho N_2& \frac{1}2I+K_2^\top&
\end{bmatrix}+\begin{bmatrix}
                                (K_1+K_2)   &-({\rho^{-1}}S_2 + S_1)\\[1.2ex]
                                  {N_1+{\rho}N_2}& -(K_1^\top +K^\top_2)
                                 \end{bmatrix}{\frac{1}{1+\rho}\begin{bmatrix}
{\rho} I&-2 S_\kappa\\[1.2ex]
{2\rho}  N_\kappa &  I
      \end{bmatrix}}. 
\] 
We note that the operator ${\rm GFK}^t$ is exactly the operator corresponding to the GCSIE formulations~\cite{turc2,turc3}. Clearly 
\[
  {\rm GFK}^t= {\rm GFK}_0^t +{\rm L}_3^t
\]
where ${\rm GFK}_0^t:H^s(\Gamma)\times H^{s-1}(\Gamma)\to H^s(\Gamma)\times  H^{s-1}(\Gamma)$
is invertible for $s\in[0,1]$ and ${\rm L}_3^t:H^{0}(\Gamma)\times H^{-1}(\Gamma)\to H^{1}(\Gamma)\times H^{0}(\Gamma)$ is compact. Then, 
the null space of this operator is contained in $H^1(\Gamma)\times H^0(\Gamma)$. We can then follow the arguments in \cite{turc3} and conclude   ${\rm GFK}^t$ is injective (the arguments are still valid for Lipschitz domains once we ensure that the elements of the null space are smooth enough), and therefore, from Fredholm alternative, invertible from where one derives the invertibility of our operator ${\rm GFK}$. 
\end{proof}

\begin{lemma}\label{lemma:polynomial}
 For any positive $\rho\ne 1$, the roots of the polynomial
 \[
  x^3-\frac{3 (\rho +1)}{ 2(\rho-1)}x^2+\frac{\rho +1}{2(\rho-1) }
 \]
 are all real and lie outside the interval $[-1/2,1/2]$. 
\end{lemma}
\begin{proof}
We can reduce the problem to study the polynomial
\[
 q_{\beta}(x)=x^3- {3\beta}  x^2+\beta
\]
for $\beta= {\frac{\rho+1}{2(\rho-1)}}\in\mathbb{R}\setminus[-1/2,1/2]$. 

Assume $\beta>1/2$. Then $q_\beta(\pm 1/2)=\pm \tfrac18+\tfrac{\beta}4$. That is,
$q_\beta(\pm 1/2)>0$. Also, $q_\beta(2\beta)=4\beta\left(\frac{1}{4}-\beta^2\right)<0$, and thus $q_\beta$ has a zero between $1/2$ and $2\beta$. Furthermore, since $\lim_{x\to\mp\infty}q_\beta(x)=\mp\infty$, it follows that $q_\beta$ must have  one (real) zero  in $(-\infty,-1/2)$, and another zero in $(2\beta,\infty)$. Hence, we have shown that the three roots of the polynomial are real and lie outside the interval $[-1/2,1/2]$.


For $\beta<-1/2$, one can show, similarly, that $q_\beta$, has three real roots, two of them  smaller than $-1/2$, and the other one larger than  $1/2$. 
\end{proof}

We end this section studying the SCFIE formulation \eqref{eq:single} which again turns out to be stable, although in weaker norms.
\begin{theorem}\label{well_posedness_SIE}
The operator ${\rm SIE}:H^s(\Gamma)\to H^s(\Gamma)$  associated to the formulation {\rm SCFIE}~\eqref{eq:single} is invertible for any $s\in[-1,0]$
\end{theorem}
\begin{proof}
It can be easily seen that if $\rho \neq 1$ we have 
\begin{eqnarray*}
 {\rm SIE}&=& 2(1-\rho) \Big(-\frac{1+\rho}{4(1-\rho)}I-\frac{\rho}{1-\rho}  K_0^\top+\Big(K_0^\top\Big)^2\Big)+{\rm L}_4\\
 &=& 2(1-\rho)\Big(\tfrac12 I+K_0^\top\Big)\left(-\tfrac{\rho+1}{2(\rho-1)}I+K_0^\top\right)+{\rm L}_4
\end{eqnarray*}
where ${\rm L}_4:H^{-1}\to H^0$ can be checked to be  compact.  We have that the operator $\frac{1}{2}I+K_0^\top$   is  Fredholm of index zero (cf. Theorem \ref{theo:inv}) while the operator $-\frac{\rho+1}{2(\rho-1)}I+K_0^\top$ is invertible,  and thus the main part of the ${\rm SIE}$ operator is Fredholm of index zero as well
which is enough for our purposes. 
Fredholm alternative implies now the invertibility of ${\rm SIE}$. (We refer to the seminal paper
\cite{KleinmanMartin} for a proof of the injectivity). The case $\rho=1$ can be treated similarly.
\end{proof}

\section{High-order Nystr\"om methods for the discretization of the formulations   CFIEFK~\eqref{eq:system_trans_FK}, CFIESK~\eqref{eq:system_trans}, CFIER~\eqref{eq:GFK}, and SCFIE~\eqref{eq:single} \label{singular_int}}

\parskip 10pt plus2pt minus1pt
\parindent0pt

We present in this section Nystr\"om discretizations of the formulations CFIEFK~\eqref{eq:system_trans_FK}, CFIESK~\eqref{eq:system_trans}, CFIER~\eqref{eq:GFK2}, and SCFIE~\eqref{eq:single} presented in the previous section.  The discretizations of some of these formulations,  for smooth curves, have been already analysed in \cite{Boubendir23032015}. 
The key component for polygonal domains  is to use sigmoidal-graded meshes that accumulate points polynomially at corners and to reformulate the aforementioned systems of integral equations in terms of {\em more\ regular\ densities} and weighted versions of the boundary integral operators of Helmholtz equations.

\subsection{Weighted boundary integral operators for Helmholtz equation}

We assume that the domain $D$ 
has corners at $\mathbf{x}_1,\mathbf{x}_2,\ldots,\mathbf{x}_P$ whose apertures measured inside $D$ are respectively $ {\theta_1,\theta_2,\ldots,\theta_P}$, and that $\Gamma\setminus\{\mathbf{x}_1,\mathbf{x}_2,\ldots,\mathbf{x}_P\}$ is piecewise analytic. 
Let  $(x_1(t),x_2(t))$ be a  $2\pi$ periodic parametrization of $\Gamma$  so that each of the curved segments  $[\mathbf{x}_j,\mathbf{x}_{j+1}]$ is mapped by $(x_1(t),x_2(t))$ with $t\in[T_j,T_{j+1}]$. We assume that $x_1(t),x_2(t)$ are continuous and that on each interval $[T_j,T_{j+1}]$  are smooth with
$(x_1'(t))^2+(x_2'(t))^2>0$ (the one-sided derivatives are taken for $t=T_j,T_{j+1}$). Consider  
the sigmoid transform introduced by Kress in~\cite{KressCorner} 
\begin{eqnarray}\label{eq:cov_w}
w(s)&=&\frac{T_{j+1}[v(s)]^p+T_j[1-v(s)]^p}{[v(s)]^p+[1-v(s)]^p},\ T_j\leq s\leq T_{j+1},\ 1\leq j\leq P\\
v(s)&=&\left(\frac{1}{p}-\frac{1}{2}\right)\left(\frac{T_{j}+T_{j+1}-2s}{T_{j+1}-T_j}\right)^3+\frac{1}{p}\ \frac{2s-T_j-T_{j+1}}{T_{j+1}-T_j}+\frac{1}{2}\nonumber
\end{eqnarray} 
where $p\geq 2$.  The function $w$ is a smooth, increasing, bijection on each of the intervals $[T_j,T_{j+1}]$ for $1\leq j\leq P$, with $w^{(k)}(T_j)=w^{(k)}(T_{j+1})=0$
for $1\leq k\leq p-1$. We then define the new parametrization
\[
 x(t)=(x_1(w(t)),x_2(w(t)))
\]
extended by $2\pi-$periodicity, if needed, to any $t\in\mathbb{R}$.


A central issue that collocation discretizations of the integral equations CFIEFK~\eqref{eq:system_trans_FK}, CFIESK~\eqref{eq:system_trans}, and SCFIE~\eqref{eq:single} is confronted with is the possibly unbounded nature of their solutions at corners. In what follows we deal with this issue by introducing new {\em weighted} unknown densities defined as
\begin{equation}\label{eq:unknowns1} 
 {\gamma_N^{1,w} u^t}(t):=(\gamma_N{u^t})(\mathbf{x}(t))\ |\mathbf{x}'(t)|
\end{equation}
for equations CFIEFK~\eqref{eq:system_trans_FK} and CFIESK~\eqref{eq:system_trans}, as well as a new {\em weighted} unknown density defined as
\begin{equation}\label{eq:weightedmu}
\mu^w(t):=\mu(\mathbf{x}(t))|\mathbf{x}'(t)|
\end{equation}
for the equation SCFIE~\eqref{eq:single}.

According to the classical theory of singularities of solutions of elliptic problems in non-smooth domains, the solutions of the integral equations CFIEFK~\eqref{eq:system_trans_FK}, CFIESK~\eqref{eq:system_trans}, CFIER~\eqref{eq:GFK}, and SCFIE~\eqref{eq:single} exhibit corner singularities~\cite{costabel-stephan}. In the case of smooth incident fields $u^{inc}$ (e.g. plane wave incidence) we have that $({\gamma_D u^{inc}}, {\gamma_N u^{inc}})\in H^1(\Gamma)\times L^2(\Gamma)$. Given that by Theorem~\ref{well-posedness-CFIEFK} the operators ${\rm CSK}$ (and/or ${\rm CFK}$) are invertible in the space $H^1(\Gamma)\times L^2(\Gamma)$ we obtain that $(\gamma_D{u^t}, \gamma_N{u^t})\in H^1(\Gamma)\times L^2(\Gamma)$. Consequently, it follows from Sobolev embedding theorems that $\gamma_D{u^t}$ is H\"older continuous on $\Gamma$. Also, given that by Theorem~\ref{well_posedness_SIE} the operators ${\rm SIE}$ are invertible in $L^2(\Gamma)$, we obtain that the solution $\mu$ of the SCFIE equation~\eqref{eq:single} belongs to $L^2(\Gamma)$. In conclusion, in the case of smooth incident fields $u^{inc}$, the $L^2$ integrable corner singularities of $\gamma_N{u^t}$ and $\mu$ are mollified by the corner polynomially vanishing weights $|\mathbf{x}'|$ for large enough values of the exponent $p$ of the sigmoid transform. Indeed, for large enough values of $p$ the $2\pi$ periodic functions $ {\gamma_N^{1,w} u^t}$ and $\mu^w$ vanish at $T_j$ for all $1\leq j\leq P$ and are regular enough.

\begin{remark}
More precise statements can be made about the nature of the functions $ \gamma_N^{1,w} u^t $ and $\mu^w$ if we resort to the corner asymptotic behavior of $ \gamma_N^{1,w} u^t $. Indeed, it was shown in~\cite{costabel-stephan} that under the assumptions of smooth incident fields and $\rho\neq 1$ the Neumann trace $\gamma_N^{1,w} u^t$ behaves as $\gamma_N^{1,w} u^t \sim c_j r_j^{\lambda_j},\ r_j\to 0,\ -1<\lambda_j<0,\ c_j\in \mathbb{C}$ where $r_j$ denotes the radial distance to the corner $\mathbf{x}_j$ and $\lambda_j$ is a solution of the transcendental equation
\begin{equation}\label{eq:sing_exp}
\frac{\sin(\lambda_j\pi-(1+\lambda_j) {\theta_j})}{\sin(\lambda_j\pi)}=\mp\frac{\rho+1}{\rho-1}
\end{equation}
where ${\theta_j}$ is the aperture of the interior angle at the corner $\mathbf{x}_j$. In the case of smooth incident fields and $\rho=1$, it can be shown that $\gamma_N^{1,w} u$ is actually more regular. Similar arguments allow us to obtain identical exponents in the corner asymptotic behavior of the solution $ \mu^{w} $ of the SIE equations~\eqref{eq:single}. 
\end{remark}

In what follows we introduce the graded-parameterized version of the four boundary integral operators of the Helmholtz equation. In the light of the discussion above on the regularity of $\gamma_D u$, $ \gamma_N^{1,w} u^t $, and $\mu^w$, we consider the cases when these boundary integral operators act on two types of $2\pi$ periodic densities: (1) we assume that $\varphi\in C^\alpha[0,2\pi]$ where $\alpha$ is large enough and in addition $\varphi(t)$ behaves like $|t-T_j|^r,\ r>0$ for all $1\leq j\leq P+1$; and (2) we assume that $\psi\in C^{0,\beta}[0,2\pi], 0<\beta<1$ is a H\"older continuous density. We start by defining the parametrized single layer operator in the form
\begin{equation}\label{eq:sl_w}
(S_k\varphi)(t):=
\int_0^{2\pi} G_k(\mathbf x(t) -\mathbf x(\tau))\varphi(\tau)d\tau.
\end{equation}
We define next the parametrized double layer operator in the form
\begin{equation}\label{eq:k_w}
(K_k\psi)(t):=\int_0^{2\pi}\frac{\partial G_k(\mathbf x(t)-\mathbf x(\tau))}{\partial\mathbf{n}(\mathbf x(\tau))}|{\bf x}'(\tau)| \psi(\tau)d\tau.
\end{equation}
and the parametrized weighted adjoint of the double layer operator as
\begin{equation}\label{eq:kt_w}
(K_k^{\top,w}\varphi)(t):=\int_0^{2\pi}|{\bf x}'(t)|\frac{\partial G_k(\mathbf x(t)-\mathbf x(\tau))}{\partial\mathbf{n}(\mathbf x(t))} \varphi(\tau)d\tau  
\end{equation}
Finally, we defined the parametrized weighted hypersingular operator as
\begin{equation}\label{eq:n_w}
\begin{split}
(N_k^w\psi)(t):=&
k^{2}\int_0^{2\pi} G_k(\mathbf x(t) -\mathbf x(\tau))
|\mathbf{x}'(t)|\ |\mathbf{x}'(\tau)| (\mathbf{n}(\mathbf x(t))\cdot\mathbf{n}(\mathbf x(\tau)))\psi(\tau)d\tau\\
&+ {\rm PV}
\int_\Gamma |\mathbf{x}'(t)|(\partial_s G_k)(\mathbf x(t) -\mathbf x(\tau)) \psi'(\tau)d\tau.
\end{split}
\end{equation} 
Having defined parametrized weighted versions of the boundary integral operators associated with the Helmholtz equation, we present next parametrized weighted versions of the integral equations  CFIEFK~\eqref{eq:system_trans_FK}, CFIESK~\eqref{eq:system_trans}, CFIER~\eqref{eq:GFK}, and SCFIE~\eqref{eq:single}.

In the case of the {\em physical} formulations CFIEFK~\eqref{eq:system_trans_FK} and CFIESK~\eqref{eq:system_trans}, we introduce the additional notation (see~\eqref{eq:unknowns1})
\begin{equation}\label{eq:D_trace} 
 \gamma_D u(t):=(\gamma_Du)(\mathbf{x}(t)),\quad \gamma_D u^{inc}(t):=u^{inc}(\mathbf{x}(t)),\quad  {\gamma_N^{w} u^{inc}}(t):=(\gamma_N u^{inc}) (\mathbf{x}(t)) |\mathbf{x}'(t)|,
\end{equation}
We multiply both sides of the second equations in formulations CFIEFK~\eqref{eq:system_trans_FK} and CFIESK~\eqref{eq:system_trans} by the term $|\mathbf{x}'(t)|$ and we obtain
\begin{equation}\label{eq:system_trans_FK_w}
{\rm CFK}^w 
\begin{bmatrix}
 \gamma_D u^t\\\gamma^{1,w}_N u^t
      \end{bmatrix}x
      :=
\begin{bmatrix}
  -(K_1+K_2) & (\rho^{-1} S_2 + S_1) \\
  -(N_1^w +\rho N_2^w)&(K^{\top,w}_1+K^{\top,w}_2)
  \end{bmatrix}  
\begin{bmatrix}
 \gamma_D u^t\\\gamma^{1,w}_N u^t
      \end{bmatrix} 
= \begin{bmatrix}  
   \gamma_D  u^{inc}\\
  \gamma^w_N u^{inc}
    \end{bmatrix}
\end{equation}
and respectively
\begin{equation}
\label{eq:system_trans_w}
\begin{split}
{\rm CSK}^w 
\begin{bmatrix}
 \gamma_D u^t\\\gamma^{1,w}_N u^t
      \end{bmatrix} 
&:=
      \Bigg(\frac{\rho^{-1}+1}{2}\begin{bmatrix}
                                 I&\\
                                 &I
                                \end{bmatrix}
                                +\begin{bmatrix}
                                   (K_2-\rho^{-1}K_1)&\rho^{-1} (S_1 - S_2)\\
                                   -(N_1^w - N_2^w)&(K_1^{\top,w}-\rho^{-1} K^{\top,w}_2)
                                 \end{bmatrix}
\Bigg) 
\begin{bmatrix}
 \gamma_D u^t\\\gamma^{1,w}_N u^t
      \end{bmatrix}
       \\
      &= \begin{bmatrix}  
   \rho^{-1}\gamma_D  u^{inc}\\
  \gamma^w_N u^{inc}
    \end{bmatrix} 
    \end{split}
\end{equation}
The weighted version of the CFIER formulations~\eqref{eq:GFK} can be written as
\begin{eqnarray}
\label{eq:GFKw}{\rm GFK}^w} 
{
\begin{bmatrix} 
 \gamma_D u^t\\ \gamma^{1,w}_N u^t 
      \end{bmatrix} &:=&\bigg(\frac{\rho}{\rho+1}  {\rm CSK}^w +\frac{2}{1+\rho}\begin{bmatrix}
      &S_{\kappa}\nonumber\\
      -\rho N_{\kappa}^w
      \end{bmatrix} {\rm CFK}^w \bigg)\begin{bmatrix}
 \gamma_D u^t\\ \gamma^{1,w}_N u^t 
      \end{bmatrix}\\
&=&\frac{1}{\rho+1}\begin{bmatrix}
      I&2S_{\kappa}\\
      -2\rho N_\kappa^w&\rho I
      \end{bmatrix} 
      \begin{bmatrix} 
   \gamma_D  u^{inc}\\
  \gamma^w_N u^{inc}
    \end{bmatrix} ,
 \quad 0\leq t<2\pi.
\end{eqnarray}
In addition, with efficiency considerations in mind, we also consider an alternative version of the regularized formulations in which the operators $S_\kappa$ and $N_\kappa^w$ are replaced by appropriate Fourier multipliers which are the principal symbols of the former operators in the sense of pseudodifferential operators~\cite{turc2}. Specifically, we use the following Fourier multipliers
\begin{equation}\label{eq:defPS1w}
(PS^w_{N,\kappa}\phi)(\mathbf{x}(t))=|\mathbf{x}'(t)|\sum_{n\in\mathbb{Z}}\sigma_{N,\kappa}(n)\hat{\phi}(n)e^{int},\quad \sigma_{N,\kappa}(\xi)=-\frac{1}{2}\sqrt{|\xi|^2-\kappa^2}
\end{equation}
and
\begin{equation}\label{eq:defPS2w}
(PS_{S,\kappa}\psi)(\mathbf{x}(t))=\sum_{n\in\mathbb{Z}}\sigma_{S,\kappa}(n)\hat{\psi}(n)e^{int},\quad \sigma_{S,\kappa}(\xi)=\frac{1}{2\sqrt{|\xi|^2-\kappa^2}}
\end{equation}
acting on $2\pi$-periodic densities $\phi$ and $\psi$, where $\hat{\phi}(n)$ and $\hat{\psi}(n)$ are the Fourier coefficients of the functions $\phi$ and $\psi$ respectively. With the aid of these Fourier multipliers we define the weighted Principal Symbol CFIER formulation (CFIERPS) 
\begin{eqnarray}
\label{eq:GFKwPS}
 {\rm PSGFK}^w   \begin{bmatrix} 
 \gamma_D u^t\\
\  \gamma^{1,w}_N u^t 
      \end{bmatrix}  &:=&\bigg(\frac{\rho}{\rho+1} {{\rm CSK}^w}+\frac{2}{1+\rho}\begin{bmatrix}
      &PS_{S,\kappa}\nonumber\\
      -\rho PS_{N,\kappa}^w
      \end{bmatrix} \begin{bmatrix} 
 \gamma_D u^t\\
\  \gamma^{1,w}_N u^t 
      \end{bmatrix} \\
&=&\frac{1}{\rho+1}\begin{bmatrix}
      I&2PS_{S,\kappa}\\
      -2\rho PS_{N,\kappa}^w&\rho I
      \end{bmatrix} 
      \begin{bmatrix}
   \gamma_D  u^{inc}\\
  \gamma^{w}_N u^{inc}
      \end{bmatrix} , \quad 0\leq t<2\pi.
\end{eqnarray}
We note that we did not establish the well-posedness of the CFIERPS formulations. Nevertheless, we give plenty numerical evidence in Section~\ref{numres} that the CFIERPS formulations are robust and computationally advantageous.

Finally, in order to derive weighted parametrized versions of the SCFIE equation, we consider the same layer representation of the fields $u^2$ and $u^1$, but a weighted  Burton-Miller type combination of the form 
 $\big(\gamma_N^{1,w} u^1 -i\eta \gamma_D^{1} u^1\big)-\big(\rho \gamma_N^{2,w} u^2-i\eta \gamma_D^{2} u^2\big)=-
\gamma_N^{w} u^{inc} +i\eta \gamma_D u^{inc}$ ($\gamma_N^{j,w} u^j$ for $j=1,2$ are defined as in \eqref{eq:unknowns1}) to lead to the following weighted boundary integral equation  
\begin{equation}\label{eq:single_w}
-\frac{1+\rho}{2}\mu^w+\mathbf{K}^w\mu^w-i\eta\mathbf{S}^w\mu^w=\gamma_N^w u^{inc}-i\eta \gamma_D u^{inc},\quad \eta\in\mathbb{R}\,\quad \eta\neq 0,\quad {\rm on}\ [0,2\pi]
\end{equation}
where
$$\mathbf{K}^w=-K_2^{\top,w}(\rho I-2K_2^{\top,w})-\rho K_1^{\top,w}(I+2K_2^{\top,w})+2(N_1^{w}-N_2^{w})S_2$$
and
$$\mathbf{S}^{w}=-\rho S_1(I+2K_2^{\top,w})-(I-2K_1)S_2.$$

\begin{remark}
It is possible to define more general weighted unknowns if we replace the weight $|\mathbf{x}'|$ by $|\mathbf{x}'|^ {\delta},\ \delta\geq 1/2$, the new weights leading to new weighted integral equation formulations. Furthermore, it is possible to define weighted unknowns that involve the Dirichlet data, that is ${\gamma}_D^wu(t):=u(\mathbf{x}(t)) |\mathbf{x}'(t)|^{\delta}$ leading again to new weighted integral equations. We note that the latter type of weighted unknowns (with $ {\delta}=1/2$) and weighted boundary integral equations of CFIESK type were used in~\cite{greengard1,greengard2}.
\end{remark}

\subsection{Nystr\"om discretizations based on kernel splitting and trigonometric interpolation}

We use a Nystr\"om discretization of the weighted parametrized equations~\eqref{eq:system_trans_FK_w},~\eqref{eq:system_trans_w},~\eqref{eq:GFKw},~\eqref{eq:GFKwPS}, and~\eqref{eq:single_w} that relies on (a) splitting of the kernels of the weighted parametrized operators into smooth and singular components, (b) trigonometric interpolation of the unknowns of these integral equations, and (c) analytical expressions for the integrals of products of periodic singular and weakly singular kernels and Fourier harmonics. Several details of this method were originally introduced in~\cite{LuW}; for completeness, we give out the full details in what follows. 

We present first a strategy to split the kernels of the weighted parametrized integral operators featured in equations~\eqref{eq:system_trans_FK_w},~\eqref{eq:system_trans_w},~\eqref{eq:GFKw}, and~\eqref{eq:single_w} into smooth and singular components.  The latter, in turn, can be expressed themselves as products of {\em known} singular kernels and smooth kernels.  We assume that $k>0$ and we begin by looking at the operator 
\begin{equation}\label{eq:sl:b}
(S_k\varphi)(t)=\int_0^{2\pi} M_k(t,\tau)\varphi(\tau)d\tau:=
\int_0^{2\pi} G_k(\mathbf x(t) -\mathbf x(\tau))\varphi(\tau)d\tau,
\end{equation} 
where $\varphi$ it is a sufficiently smooth $2\pi-$periodic function---recall that basically $\varphi(t) := |\mathbf{x}'(t)| \gamma_N^1u(\mathbf{x}(t))$. The kernel $M_k(t,\tau)$ can be expressed in the form
\[
M_k(t,\tau)= M_{k,1}(t,\tau)\ln\left(4\sin^2\frac{t-\tau}{2}\right)+M_{k,2}(t,\tau)
\]
with
\begin{eqnarray*}
M_{k,1}(t,\tau)&:=&-\frac{1}{4\pi}J_0(k|\mathbf{r}|)\\
M_{k,2}(t,\tau)&:=&M_k(t,\tau)-M_{k,1}(t,\tau)\ln\left(4\sin^2\frac{t-\tau}{2}\right)
\end{eqnarray*} 
where we have denoted, for lightening this and forthcoming expressions, 
\[
 {\bf r}={\bf r}(t,\tau)= {\bf x}(t)-{\bf x}(\tau) .
\]
Observe that the diagonal terms are given by
\[
M_{k,1}(t,t)=-\frac{1}{4\pi},\qquad M_{k,2}(t,t)=\frac{i}{4}-\frac{C}{2\pi}-\frac{1}{2\pi}\ln\frac{k|\mathbf{x}'(t)|}{2}
\]
where $C  \approx 0.5772$ is the Euler-Mascheroni constant.

The parametrized double layer operator, see \eqref{eq:double}, is defined as follows 
\begin{equation}\label{eq:k:b}
(K_k\psi)(t)=\int_0^{2\pi} H_k(t,\tau)\psi(\tau)d\tau:=\int_0^{2\pi}\frac{\partial  G_k( \mathbf x(t)-\mathbf x(\tau) ) }{\partial\mathbf{n}(\mathbf x(\tau))}|{\bf x}'(\tau)| \psi(\tau)d\tau.
\end{equation}
We note that the kernel of the operator $K_k$ behaves as (i) $|t-\tau|^{-1}$ when $t\to T_j, t<T_j$ and $\tau\to T_j,\ \tau>T_j$ for $2\leq j\leq P$ and as (ii) $(|t-\tau|\mod 2\pi)^{-1}$ when $t\to T_1=0$ and $\tau\to T_{P+1}=2\pi$ (that is when $\mathbf{x}(t)$ and $\mathbf{x}(\tau)$ approach a corner from different sides). Thus, the integral in the definition of the operator $K_k$ should be understood in the sense of Cauchy Principal Value integral.  However, for H\"older continuous densities $\psi$, it is possible to recast the operators $K_k$ into an equivalent form that features operators that involve integrable expressions only~\cite{KressCorner}. In order to do so, we 
 express $K_k$ in the form
\begin{eqnarray}\nonumber
(K_k\psi)(t)&=&\int_0^{2\pi}\frac{\partial}{\partial\mathbf{n}(\mathbf x(\tau))} \big[ G_k( \mathbf x(t)-\mathbf x(\tau) ) - G_0( \mathbf x(t)-\mathbf x(\tau) ) \big]|{\bf x}'(\tau)| \psi(\tau)d\tau\\
&+&\int_0^{2\pi}\frac{ \partial G_0( \mathbf x(t)-\mathbf x(\tau) )}{\partial\mathbf{n}(\mathbf x(\tau))}|{\bf x}'(\tau)| (\psi(\tau)-\psi(t))d\tau\nonumber\\
&+&\psi(t)\int_0^{2\pi}\frac{\partial  G_0( \mathbf x(t)-\mathbf x(\tau) ) }{\partial\mathbf{n}(\mathbf x(\tau))}|{\bf x}'(\tau)|d\tau.\label{eq:Kkparam}
\end{eqnarray}

We note first that
\[
\int_0^{2\pi}\frac{ \partial G_0( \mathbf x(t)-\mathbf x(\tau) )}{\partial\mathbf{n}(\mathbf x(\tau))}|{\bf x}'(\tau)|d\tau=\begin{cases}\displaystyle-\frac{1}{2}&{\rm if}\ t\in [0,2\pi]\setminus \{T_1,\ldots,T_P\}\\[1.4ex]
\displaystyle-\frac{1}{2\pi} \theta_j &{\rm if}\ t=T_j, 1\leq j\leq P,\end{cases}
\] 
where $\theta_j$ is the inner angle of the $j$th corner. 
For the second integral, and since we have assumed $\psi$ to be H\"older continuous, the integrand is weakly singular. Finally, in the kernel  of the  first integral operator  in \eqref{eq:Kkparam} we find the function 
\[
 H_k(t,\tau)=\left[
\frac{\partial G_k}{\partial\mathbf{n}(\mathbf x(\tau))}\right] ({\bf r})\:|{\bf x}'(\tau)|= 
 \frac{ik}{4} \bm{\nu}(\tau)\cdot\mathbf{r}\ \frac{H_1^{(1)}(k|\mathbf{r}|)}{|\mathbf{r}|},
\]
where 
\[
{ \bm{\nu}(\tau)}:={\bf n}({\bf x}(\tau))|{\bf x}'(\tau)|=w'(\tau)(x_2'(w(\tau)),-x_1'(w(\tau))).\ 
\]
We have  in addition  the decomposition
\[
H_k(t,\tau)= H_{k,1}(t,\tau)\ln\left(4\sin^2\frac{t-\tau}{2}\right)+H_{k,2}(t,\tau)
\]
where
\begin{eqnarray*}
H_{k,1}(t,\tau)&:=&-\frac{k}{4\pi} \bm{\nu}(\tau)\cdot\mathbf{r}\ \frac{J_1(k|\mathbf{r}|)}{|\mathbf{r}|}\\
H_{k,2}(t,\tau)&:=&H_k(t,\tau)-H_{k,1}(t,\tau)\ln\left(4\sin^2\frac{t-\tau}{2}\right)
\end{eqnarray*}
which have the diagonal terms
\[
H_{k,1}(t,t)=0,\qquad H_{k,2}(t,t)=\frac{1}{4\pi}\frac{ \bm{\nu}(t)\cdot \mathbf{x}''(t)}{|\mathbf{x}'(t)|^2}.
\]
It can be easily seen that the second function  in the kernel of the first integral operator in \eqref{eq:Kkparam} is given by
\[
H_0(t,\tau)=\left[\frac{\partial G_0}{\partial\mathbf{n}(\mathbf x(\tau))}\right] ({\bf r})\:|{\bf x}'(\tau)| =\frac{1}{2\pi}\frac{ \bm{\nu}(\tau)\cdot\mathbf{r}}{|\mathbf{r}|^2},\qquad H_0(t,t)=\frac{1}{4\pi}\frac{ \bm{\nu}(t)\cdot \mathbf{x}''(t)}{|\mathbf{x}'(t)|^2},
\]
and thus $H_{k,2}(t,t)-H_0(t,t)$, appearing in the first term in \eqref{eq:Kkparam}, 
can be defined even at corner points where $|\mathbf{x}'(t)|=0$. 

The graded-parametrized adjoint of the double layer cf. \eqref{eq:adj_double} is given by 
\begin{equation}\label{eq:kt:b}
(K_k^{\top,w}\varphi)(t)=\int_0^{2\pi} H^\top_k(t,\tau)\varphi(\tau)d\tau:=\int_0^{2\pi}|{\bf x}'(t)|\frac{\partial  G_k( \mathbf x(t)-\mathbf x(\tau) ) }{\partial\mathbf{n}(\mathbf x(t))} \varphi(\tau)d\tau. 
\end{equation}
Here
\[
H^\top_k(t,\tau) =\frac{ik}{4} \bm{\nu}(t)\cdot\mathbf{r}\ \frac{H_1^{(1)}(k|\mathbf{r}|)}{|\mathbf{r}|}.
\]
The kernel $H_k^\top(t,\tau)$ can be expressed in the form
\[
H_k^\top(t,\tau)= H_{k,1}^\top(t,\tau)\ln\left(4\sin^2\frac{t-\tau}{2}\right)+H_{k,2}^\top(t,\tau)
\]
with
\begin{eqnarray*}
H_{k,1}^\top(t,\tau)&:=&-\frac{k}{4\pi} \bm{\nu}(t)\cdot\mathbf{r}\ \frac{J_1(k|\mathbf{r}|)}{|\mathbf{r}|}\\
H_{k,2}^\top(t,\tau)&:=&H_k^\top(t,\tau)-H_{k,1}^\top(t,\tau)\ln\left(4\sin^2\frac{t-\tau}{2}\right)
\end{eqnarray*}
and
\[
H_{k,1}^\top(t,t)=0,\qquad H_{k,2}^\top(t,t)=\frac{1}{4\pi}\frac{ \bm{\nu}(t)\cdot \mathbf{x}''(t)}{|\mathbf{x}'(t)|^2}.
\]
A simple calculation shows that $H_{k,2}^\top(t,t)$ is infinite whenever $|\mathbf{x}'(t)|=0$, that is $w'(t)=0$.
\begin{remark}\label{vanishing}
Notice that although $H_{k,2}^\top$ is unbounded in and around corners, the product $H_{k,2}^\top(t,t)\varphi(t)$ still vanishes at corners. Indeed, in a neighborhood of a corner, say $t\sim T_j$ the function $\varphi$ behaves as $\varphi(t)\sim c_j [w'(t)]^{1+\lambda_j},\ -1/2<\lambda_j$. A careful inspection of the singularity of $H_{k,2}^\top(t,t)$ reveals that this expression behaves as $w''(t)/w'(t)$ and thus the product $H_{k,2}^\top(t,t)\varphi(t)\sim c_j w''(t)[w'(t)]^{\lambda_j}$. Given that $w(t)\sim |t-T_j|^p,\ t\to T_j$, we see that the latter product is regular enough for $t\to T_j$ provided that $p$ is  sufficiently large.
\end{remark}

Finally, for the graded-parametrized version of the hypersingular operator $N_k$, we add and subtract $\frac{1}{4\pi}\ln(4\sin^2((t-\tau)/2)$ to get
\begin{eqnarray*}
(N_k^{w} \psi)(t)&=&-\mathrm{PV}\frac{1}{4\pi}\int_0^{2\pi} \cot\frac{t-\tau}{2}\:\psi'(\tau)\,{\rm d}\tau
+\int_0^{2\pi}Q_k(t,\tau)\psi(\tau)\,{\rm d}\tau+\int_0^{2\pi}D_k(t,\tau)\psi'(\tau)\,{\rm d}\tau
\end{eqnarray*}
with
\begin{eqnarray*}
Q_k(t,\tau)\!\!&:=&\!\!\!k^{2} M_k(t,\tau)(\mathbf x'(t)\cdot\ \mathbf x'(\tau))\\
D_k(t,\tau)\!\!&:=&\!\!\!\frac{\partial}{\partial t}\left(\frac{1}{4\pi}\ln\left(\sin^2\frac{t-\tau}{2}\right)+M_k(t,\tau)\right).
\end{eqnarray*} 
Note  we have used 
\[                                              
|\mathbf x'(t)||\mathbf x'(\tau)|(\mathbf{n}(\mathbf x(t))\cdot\mathbf{n}(\mathbf x(\tau))= \mathbf x'(t) \cdot  \mathbf x'(\tau) .
  \]
The kernel $Q_k$ can be treated similarly to the kernel $M_k$. On the other hand, a simple calculation gives that
\[
D_k(t,\tau)= D_{k,1}(t,\tau)\ln\left(4\sin^2\frac{t-\tau}{2}\right)+D_{k,2}(t,\tau)
\]
where
\begin{eqnarray*}
D_{k,1}(t,\tau)&:=&\frac{k}{4\pi}\mathbf{x}'(t)\cdot\mathbf{r}\ \frac{J_1(k|\mathbf{r}|)}{|\mathbf{r}|}\\
D_{k,2}(t,\tau)&:=&D_k(t,\tau)-D_{k,1}(t,\tau)\ln\left(4\sin^2\frac{t-\tau}{2}\right)
\end{eqnarray*}
have diagonal terms
\[
D_{k,1}(t,t)=0,\qquad D_{k,2}(t,t)=\frac{1}{4\pi}\frac{\mathbf{x}'(t)\cdot \mathbf{x}''(t)}{|\mathbf{x}'(t)|^2}.
\]
Again, $D_{k,2}(t,t)$ is infinite at corners, but the trapezoidal rule can still be applied since that term is multiplied by $\psi'(t)$ which vanishes at the corners---this requires the same type of justification used in Remark~\ref{vanishing}. 

We note that the weighted integral equations CFIESK~\eqref{eq:system_trans_w} and SCFIE~\eqref{eq:single_w} feature the difference operator $N_1^w-N_2^w$. While this difference can be performed directly using the methodology presented above for the evaluation of the operators $N_1^w$ and $N_2^w$, a more advantageous approach relies on the methods developed by Kress in~\cite{KressH} for the evaluation of operators $N_k^w-N_0^w$ where $N_0^w$ is the weighted hypersingular operator corresponding to wavenumber $k=0$. The latter methodology consists of expressing the graded-parametrized operators,    constructed  from \eqref{eq:normal_double:0} instead, 
as
\begin{equation}\label{eq:N_k}
([N_k^w-N_0^w] \psi )(y) =-\int_0^{2\pi}( \bm{\nu}(t))^\top 
  \nabla^2(G_k-G_0)  ( {\bf x}(t)-{\bf x}(\tau) )
(t,\tau) \bm{\nu}(\tau)\psi(\tau)d\tau
\end{equation}
We have (see the proof of Lemma \ref{lemma:kernels}) that
\begin{eqnarray*}
\nabla^2(G_k-G_0)( {\bf r} )&=&-\frac{ik^2}{4}H_0^{(1)}(k|{\bf r}|)\frac{1}{|{\bf r}|^2}
{\bf r} \: {\bf r} ^{\top}+\left(\frac{i}{4} H_1^{(1)}(k|{\bf r}|)k|{\bf r}|-\frac1{2\pi}\right)
 \left(\frac{2}{|{\bf r}|^4}{\bf r}\:  {\bf r} ^{\top} - \frac{1}{|{\bf r}|^2}I\right)\\
 &=&L_{1,k,0}(t,\tau)
\ln\left(\sin^2\frac{t-\tau}2\right)+L_{2,k,0}(t,\tau)
\end{eqnarray*}
where  $I$ is the $2\times 2$ identity matrix  and 
\begin{eqnarray*}
L_{1,k,0}(t,\tau)&=&\frac{k}{4\pi}\left[ \frac{J_1(k|\mathbf{r}|)}{|\mathbf{r}|} I +\frac{1}{|\mathbf{r}|^2}\left(kJ_0(k|\mathbf{r}|)-2\frac{J_1(k|\mathbf{r}|)}{|\mathbf{r}|}\right)\mathbf{r} \:\mathbf{r}^\top\right]\\
 {L}_{2,k,0} (t,\tau)&:=&  {\nabla^2(G_k-G_0)( {\bf r} )}-L_{1,k,0}(t,\tau)
\ln\left(\sin^2\frac{t-\tau}2\right)
\end{eqnarray*}
satisfies
\[
 L_{1,k,0}(t,t)=\frac{k^2}{8\pi} I,\quad
  L_{2,k,0}(t,t)=k^2\left[\frac{1}{4\pi}\ln\left(\frac{k|\mathbf{x}'(t)|}{2}\right)-\frac{i}{8}+\frac{2C-1}{8\pi}\right] I +\frac{k^2}{4\pi}\frac{1}{|\mathbf{x}'(t)|^2}\:
   \mathbf{x}'(t) \  (\mathbf{x}'(t))^\top 
\]
It follows then that 
\[
( \bm{\nu}(t))^\top( {\nabla^2(G_k-G_0)( {\bf r} )}) \bm{\nu}(\tau) =L_{1,k}(t,\tau)\ln\left(4\sin^2\frac{t-\tau}{2}\right)+L_{2,k}(t,\tau)
\]
with diagonal terms
\[
L_{1,k}(t,t)=\frac{k^2}{8\pi}|\mathbf{x}'(t)|^2 \qquad L_{2,k}(t,t)=k^2\left[\frac{1}{4\pi}\ln\left(\frac{k|\mathbf{x}'(t)|}{2}\right)-\frac{i}{8}+\frac{2C-1}{8\pi}\right]|\mathbf{x}'(t)|^2
\]
that are bounded even around corner points where $\mathbf{x}'(t)=0$. Thus, we can apply the procedure above for the graded-parametrized operator $N_{k_1}^w-N_{k_2}^w=(N_{k_1}^w-N_0^w)-(N_{k_2}^w-N_0^w)$  so that we  are led to integral operators whose kernels are of the form
$$L_{k_1,k_2}(t,\tau):=[L_{1,k_1}(t,\tau)- {L_{1,k_2}(t,\tau)}]\ln\left(4\sin^2\frac{t-\tau}{2}\right)+[ {L_{2,k_1}(t,\tau)-L_{2,k_2}(t,\tau)}].$$

The splitting techniques presented above can be adapted for the evaluation of the operators $ {S_\kappa}$ and $N_\kappa^w$ with $\Im(\kappa)>0$ using additional smooth cutoff function supported in neighborhoods of the target points $t$ according to the procedures introduced in~\cite{turc2}.

\subsection{Trigonometric interpolation}

We describe next a Nystr\"om method based on trigonometric interpolation that 
follows closely the quadrature method introduced by Kress in~\cite{KressH}, 
which in turn relies on the logarithmic quadrature methods introduced by 
Kussmaul~\cite{kusmaul} and Martensen~\cite{martensen}. The main idea is to use global trigonometric interpolation of the quantities  $\gamma_D u^t$, $\gamma_N^{1,w} u^t$ , and $\mu^w$ that are the solutions of the integral equations~\eqref{eq:system_trans_FK_w},~\eqref{eq:system_trans_w}, and~\eqref{eq:single_w}. Given that the larger the exponent $p$ of the sigmoidal transform is, the smoother the quantities $ {\gamma_N^{1,w} u^t}$ and $\mu^w$ are, the trigonometric interpolants of these quantities converge fast with respect to the number of  interpolation points. We choose an equi-spaced splitting of the interval $[0,2\pi]$ into $2n$ points.  We choose $T_j$ such that $T_{j+1}-T_j$ are proportional (with the same constant of proportionality) to the lengths of the arcs of $\Gamma$ from $\mathbf{x}_j$ to $\mathbf{x}_{j+1}$ for all $j$. Consequently, the number of discretization points per subinterval $[T_j,T_{j+1}],\ 1\leq j\leq P$ may differ from each other. We thus consider the equi-spaced collocation points $\{t_0^{(n)},t_1^{(n)},\ldots,t_{2n-1}^{(n)}\}$ and  the interpolation problem with respect to these nodal points
in the space $\mathbb{T}_n$ of trigonometric 
polynomials of the form
$$v(t)=\sum_{m=0}^n a_m\cos{mt}+\sum_{m=1}^{n-1}b_m\sin{mt}$$
 which  is uniquely solvable~\cite{Kress}. We denote by $P_n:C[0,2\pi]\to 
\mathbb{T}_n$ the corresponding trigonometric polynomial interpolation operator . We use the quadrature rules~\cite{KressH}
\begin{eqnarray}\label{eq:quad1}
\int_0^{2\pi}\ln\left(4\sin^2\frac{t-\tau}{2}\right)f(\tau)d\tau&\approx&\int_0^
{2\pi}\ln\left(4\sin^2\frac{t-\tau}{2}\right)(P_nf)(\tau)d\tau\nonumber\\
&=&\sum_{i=0}^{2n-1}R_i^{(n)}(t)f(t_i^{(n)})
\end{eqnarray}
where the expressions $ R_i^{(n)}(t) $ are given by
$$R_i^{(n)}(t)=-\frac{2\pi}{n}\sum_{m=1}^{n-1}\frac{1}{m}\cos{m(t-t_i^{(n)})}
-\frac{\pi}{n^2}\cos{n(t-t_i^{(n)})}.$$
We also use the trapezoidal rule
\begin{equation}\label{eq:quad2}
\int_0^{2\pi}f(\tau)d\tau\approx\int_0^{2\pi}(P_nf)(\tau)d\tau=\frac{\pi}{n}
\sum_{i=0}^{2n-1}f(t_i^{(n)}).
\end{equation} 
Finally, we have  the quadrature rule~\cite{KressH}
\begin{eqnarray}\label{eq:quad3}
\frac{1}{4\pi}\int_0^{2\pi}\cot{\frac{\tau-t}{2}}f'(\tau)d\tau&\approx&\frac{1}{4\pi}\int_0^
{2\pi}\cot{\frac{\tau-t}{2}}\frac{d}{d\tau}\left[(P_{n}f)(\tau)\right]d\tau\nonumber\\
&=&\sum_{i=0}^{2n-1}T_i^{(n)}(t)f(t_i^{(n)})
\end{eqnarray}
where 
$$T_i^{(n)}(t)=-\frac{1}{2n}\sum_{m=1}^{n-1}m\cos{m(t-t_i^{(n)})}-\frac{1}{4}\cos n(t-t_i^{(n)}).$$
The derivatives of the densities needed for the evaluation of the operators $N_k^w$ and $N_\kappa^w$ are effected by differentiation of the global trigonometric interpolant of the densities. This can be pursued either by means of Fast Fourier Transforms (FFTs) or using the Fourier differentiation matrix $D^{(n)}$ whose entries are given by $D^{(n)}(i,j)=\frac{1}{2}(-1)^{i+j}\cot\left(\frac{(i-j)\pi}{n}\right),\ i\neq j$ and $D^{(n)}(i,i)=0$. 

Finally, given that the values of $ \gamma_N^{1,w} u^t $ and $\mu^w$ vanish at corner points, the terms in the boundary integral equations that feature these quantities are not collocated at corner points. Alternatively, this issue can be bypassed altogether by shifting the mesh $t_{j}^{(n)}$ by $h/2$, where $h$ is the meshsize. All of the interpolatory quadratures presented above still apply for the shifted meshes. Finally, the Fourier multipliers $PS_{S,\kappa}^w$ and $PS_{N,\kappa}^w$ defined in equations~\eqref{eq:defPS1w} and~\eqref{eq:defPS2w} can be easily evaluated using trigonometric interpolation and FFTs.

\subsection{Numerical results}\label{numres}

We present in this section a variety of numerical results that
demonstrate the properties of the various formulations considered in this text. Solutions of the linear systems arising from the Nystr\"om discretizations of the transmission integral equations described in Section~\ref{singular_int} are obtained by means of the fully complex, unrestarted version of the iterative solver GMRES~\cite{SaadSchultz}.  For the case of the regularized formulations we present choices of the complex wavenumber $\kappa$ in each of the cases considered; our extensive numerical experiments suggest that these values of $\kappa$  lead  to nearly optimal numbers of GMRES iterations to reach desired (small) GMRES relative residuals. We also present in each table the values of the GMRES relative residual tolerances used in the numerical experiments.
 
\begin{figure}
\[
 \includegraphics[width=0.3\textwidth]{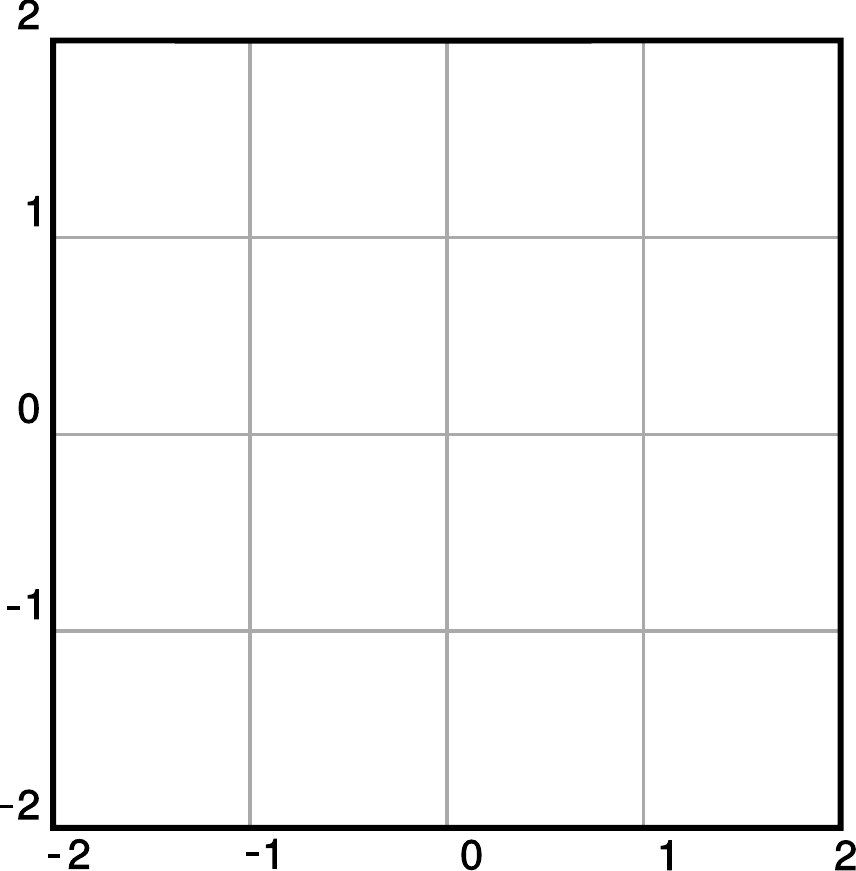}\qquad
 \includegraphics[width=0.3\textwidth]{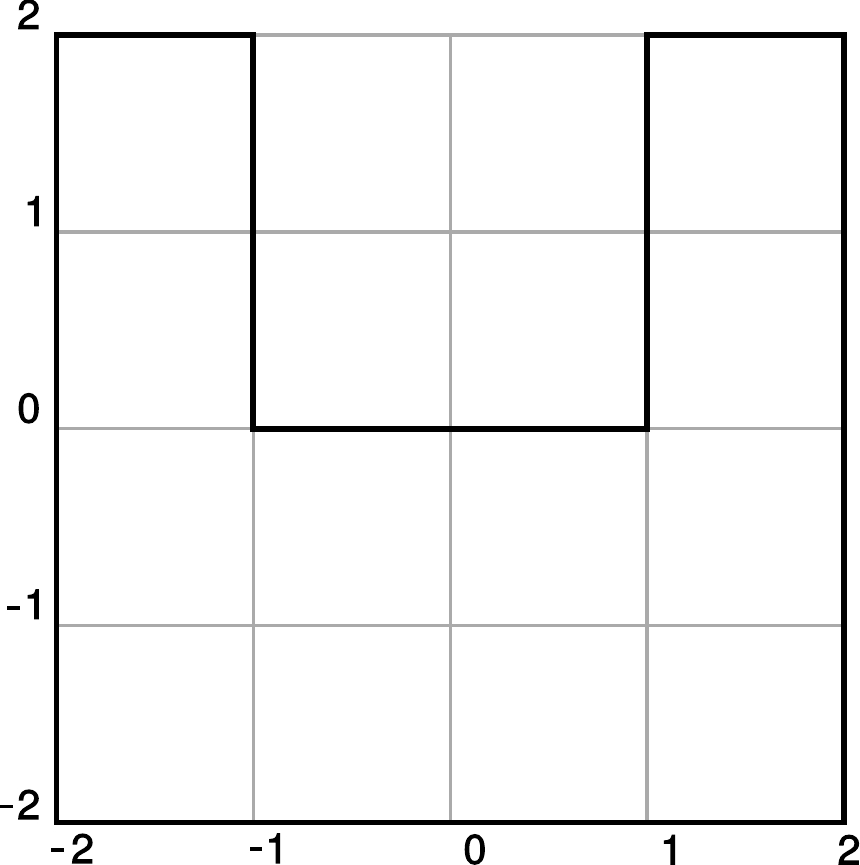}
\]
 \caption{\label{fig:U}Domains for our numerical tests: (a) square; (b) U-shape domain}
\end{figure}

We present scattering experiments concerning the following two Lipschitz geometries (see Figure \ref{fig:U}): (a) a square whose sides equal to 4, and (b) a U-shape scatterer of sides equal to 4 and indentation equal to 2 . For every scattering experiment we consider plane-wave incidence $u^{inc}$  of direction $(0,-1)$ and we present maximum far-field errors, that is we choose sufficiently many directions  $\frac{\mathbf{x}}{|\mathbf{x}|}$  (1024 directions have been used in our computations)  and for each direction we compute the far-field amplitude $u^{1}_\infty(\hat{\mathbf{x}})$ defined as
\begin{equation}
\label{eq:far_field}
u^{1}(\mathbf{x})=\frac{e^{ik_1|\mathbf{x}|}}{\sqrt{|\mathbf{x}|}}\left(u^{1}_\infty(\hat{\mathbf{x}})+\mathcal{O}\left(\frac{1}{|\mathbf{x}|}\right)\right),\
|\mathbf{x}|\rightarrow\infty.\\
\end{equation}
The maximum far-field errors were evaluated through comparisons of the
numerical solutions $u_\infty^{1,  calc}$ corresponding to either formulation with reference solutions $u_\infty^{1, ref}$ by means of the relation
\begin{equation}
\label{eq:farField_error}
\varepsilon_\infty=\max|u_\infty^{1,
calc}(\hat{\mathbf{x}})-u_\infty^{1, ref}(\hat{\mathbf{x}})|
\end{equation}
The latter solutions $u_\infty^{1, ref}$ were produced using solutions corresponding with refined discretizations based on the formulation CFIESK with GMRES residuals of $10^{-12}$ for all geometries. Besides far field errors, we display the numbers of iterations required by the GMRES solver to reach relative residuals that are specified in each case.  We used in the numerical experiments discretizations ranging from 6 to 12 discretization points per wavelength, for frequencies $k_1$ and $k_2$ in the medium to the high-frequency range corresponding to scattering problems of sizes ranging from $2.5$ to $40.8$ wavelengths. The columns ``Unknowns'' in all Tables display the numbers of unknowns used in each case, which equal to the value $4n$ defined in Section~\ref{singular_int} for the weighted CFIEFK, CFIESK, CFIER, and CFIERPS formulations, and $2n$ for the weighted SCFIE formulation. In order to remind the reader of the fact that the SCFIE formulations require half the number of unknowns required by each of the other formulations, we denote the former by SCFIE$^*$ in the tables. 
Following common practice~\cite{boubendir}, we used the CFIEFK operators as their own preconditioners, and we denote these by ${\rm CFIEFK}^2$. We note that in this case the computational time required to perform a matrix-vector product corresponding to the  ${\rm CFIEFK}^2$ formulation is double that related to the CFIEFK formulation.

We start by presenting the high-order convergence of our Nystr\"om solvers in Tables~\ref{resultsMS1}--\ref{resultsMU3}. We have used sigmoid transforms with a value $p=3$ for all formulations but SCFIE, in which case we used $p=4$. The need for a different value of $p$ for the latter formulations can be attributed to the more singular nature of the solutions of these equations. As it can be seen, solvers based on the CFIESK formulations are the most accurate on account of the facts that they do not require numerical differentiation.

\begin{table}
\begin{center}
\resizebox{!}{1.2cm}
{
\begin{tabular}{|c|c|c|c|c|c|c|c|c|c|c|c|c|}
\hline
$k_1$ & $k_2$ & Unknowns & \multicolumn{2}{c|}{CFIEFK$^2$} &\multicolumn{2}{c|}{CFIESK}&\multicolumn{2}{c|}{SCFIE$^*$}& \multicolumn{2}{c|}{CFIER}&\multicolumn{2}{c|}{CFIERPS}\\
\cline{4-13}
 & & & Iter.& $\epsilon_\infty$ &Iter.&$\epsilon_\infty$&Iter.&$\epsilon_\infty$&Iter.&$\epsilon_\infty$&Iter.&$\epsilon_\infty$\\
\hline
 1 & 4 & 256 & 32 & 6.6 $\times$ $10^{-3}$& 34 & 2.4 $\times$ $10^{-6}$ & 43 & 3.8 $\times$ $10^{-3}$ & 43 & 3.2 $\times$ $10^{-3}$& 32 & 3.1 $\times$ $10^{-3}$\\
 1 & 4 & 512 & 31 & 8.0 $\times$ $10^{-4}$ & 34 & 1.8 $\times$ $10^{-7}$ & 46 & 5.5 $\times$ $10^{-4}$ & 43 & 3.9 $\times$ $10^{-4}$ & 33 & 3.7 $\times$ $10^{-4}$ \\
 1 & 4 & 1024 & 31 & 1.0 $\times$ $10^{-4}$ & 34 & 1.1 $\times$ $10^{-8}$ & 49 & 8.0 $\times$ $10^{-5}$  & 47 & 4.8 $\times$ $10^{-5}$ & 34 & 4.6 $\times$ $10^{-5}$ \\
 1 & 4 & 2048 & 31 & 1.2 $\times$ $10^{-5}$ & 34 & 4.1 $\times$ $10^{-10}$ & 54 & 1.2 $\times$ $10^{-5}$ & 47 & 6.0 $\times$ $10^{-6}$ & 34 & 5.8 $\times$ $10^{-6}$ \\
\hline
\end{tabular}
}
\caption{\label{resultsMS1} Scattering experiments for the square
  geometry with $\rho=1$, and for the CFIEFK$^2$, CFIESK, SCFIE, and CFIERPS formulations In the SCFIE formulation we selected $\eta=k_1$. In the regularized formulations CFIER and CFIERPS we used $\kappa=(k_1+k_2)/2+i\ k_1 $. The GMRES residual was set to equal $10^{-12}$.}
\end{center}
\end{table}

\begin{table}
\begin{center}
\resizebox{!}{1.2cm}
{
\begin{tabular}{|c|c|c|c|c|c|c|c|c|c|c|c|c|}
\hline
$k_1$ & $k_2$ & Unknowns & \multicolumn{2}{c|}{CFIEFK$^2$} &\multicolumn{2}{c|}{CFIESK}&\multicolumn{2}{c|}{SCFIE$^*$}& \multicolumn{2}{c|}{CFIER}&\multicolumn{2}{c|}{CFIERPS}\\
\cline{4-13}
 & & & Iter.& $\epsilon_\infty$ &Iter.&$\epsilon_\infty$&Iter.&$\epsilon_\infty$&Iter.&$\epsilon_\infty$&Iter.&$\epsilon_\infty$\\
\hline
1 & 4 & 352 & 84 & 4.3 $\times$ $10^{-2}$& 75 & 3.7 $\times$ $10^{-4}$ & 64 & 1.4 $\times$ $10^{-2}$ & 73 &1.7 $\times$ $10^{-2}$& 62 & 1.9 $\times$ $10^{-2}$\\
 1 & 4 & 704 & 82 & 5.2 $\times$ $10^{-3}$ & 75 & 2.2 $\times$ $10^{-5}$ & 66 & 1.7 $\times$ $10^{-3}$ & 74 &2.1 $\times$ $10^{-3}$ & 63 & 2.3 $\times$ $10^{-3}$ \\
 1 & 4 & 1408 & 81 & 6.4 $\times$ $10^{-4}$ & 75 & 1.5 $\times$ $10^{-6}$ & 67 & 2.4 $\times$ $10^{-4}$ & 75 &2.6 $\times$ $10^{-4}$ & 63 & 2.9 $\times$ $10^{-4}$ \\
 1 & 4 & 2816 & 80 & 7.9 $\times$ $10^{-5}$ & 75 & 9.9 $\times$ $10^{-8}$ & 68 & 3.7 $\times$ $10^{-5}$  & 77 &3.1 $\times$ $10^{-5}$ & 63 & 3.7 $\times$ $10^{-5}$ \\
\hline
\end{tabular}
}
\caption{\label{resultsMU1} Scattering experiments for the U-shape
  geometry with $\rho=1$, and for the CFIEFK$^2$, CFIESK, SCFIE, and CFIERPS formulations In the SCFIE formulation we selected $\eta=k_1$. In the regularized formulations CFIER and CFIERPS we used $\kappa=(k_1+k_2)/2+i\ k_1 $. The GMRES residual was set to equal $10^{-12}$.}
\end{center}
\end{table}

\begin{table}
\begin{center}
\resizebox{!}{1.2cm}
{
\begin{tabular}{|c|c|c|c|c|c|c|c|c|c|c|c|c|}
\hline
$k_1$ & $k_2$ & Unknowns & \multicolumn{2}{c|}{CFIEFK$^2$} &\multicolumn{2}{c|}{CFIESK}&\multicolumn{2}{c|}{SCFIE$^*$}& \multicolumn{2}{c|}{CFIER}&\multicolumn{2}{c|}{CFIERPS}\\
\cline{4-13}
 & & & Iter.& $\epsilon_\infty$ &Iter.&$\epsilon_\infty$&Iter.&$\epsilon_\infty$&Iter.&$\epsilon_\infty$&Iter.&$\epsilon_\infty$\\
\hline
 1 & 4 & 256 & 58 & 9.9 $\times$ $10^{-4}$& 39 & 1.5 $\times$ $10^{-5}$ & 48 & 2.1 $\times$ $10^{-3}$ & 60 & 2.0 $\times$ $10^{-4}$& 76 & 4.1 $\times$ $10^{-4}$\\
 1 & 4 & 512 & 56 & 1.2 $\times$ $10^{-4}$ & 39 & 9.0 $\times$ $10^{-7}$ & 49 & 3.3 $\times$ $10^{-4}$ & 52 & 4.5 $\times$ $10^{-5}$ & 80 & 5.2 $\times$ $10^{-5}$ \\
 1 & 4 & 1024 & 54 & 1.5 $\times$ $10^{-5}$ & 37 & 6.0 $\times$ $10^{-8}$ & 51 & 5.0 $\times$ $10^{-5}$ & 57 & 6.0 $\times$ $10^{-6}$ & 84 & 6.5 $\times$ $10^{-6}$ \\
 1 & 4 & 2048 & 53 & 1.9 $\times$ $10^{-6}$ & 37 & 4.1 $\times$ $10^{-9}$ & 52 & 7.6 $\times$ $10^{-6}$ & 57 & 7.0 $\times$ $10^{-7}$ & 87 & 8.2 $\times$ $10^{-7}$ \\
\hline
\end{tabular}
}
\caption{\label{resultsMS3} Scattering experiments for the square
  geometry with $\rho=k_1^2/k_2^2$, and for the CFIEFK$^2$, CFIESK, SCFIE, and CFIERPS formulations In the SCFIE formulation we selected $\eta=k_1$. In the regularized formulations CFIER and CFIERPS we used $\kappa=(k_1+k_2)/2+i\ k_1 $. The GMRES residual was set to equal $10^{-12}$.}
\end{center}
\end{table}

\begin{table}
\begin{center}
\resizebox{!}{1.2cm}
{
\begin{tabular}{|c|c|c|c|c|c|c|c|c|c|c|c|c|}
\hline
$k_1$ & $k_2$ & Unknowns & \multicolumn{2}{c|}{CFIEFK$^2$} &\multicolumn{2}{c|}{CFIESK}&\multicolumn{2}{c|}{SCFIE$^*$}& \multicolumn{2}{c|}{CFIER}&\multicolumn{2}{c|}{CFIERPS}\\
\cline{4-13}
 & & & Iter.& $\epsilon_\infty$ &Iter.&$\epsilon_\infty$&Iter.&$\epsilon_\infty$&Iter.&$\epsilon_\infty$&Iter.&$\epsilon_\infty$\\
\hline
 1 & 4 & 352 & 110 & 6.5 $\times$ $10^{-4}$& 67 & 4.8 $\times$ $10^{-3}$ & 71 & 5.4 $\times$ $10^{-3}$ & 93 &3.5 $\times$ $10^{-4}$& 115 & 2.5 $\times$ $10^{-4}$\\
 1 & 4 & 704 & 107 & 1.0 $\times$ $10^{-4}$ & 64 & 1.1 $\times$ $10^{-3}$ & 71 & 8.0 $\times$ $10^{-4}$ & 86 &7.2 $\times$ $10^{-5}$ & 119 & 3.4 $\times$ $10^{-5}$ \\
 1 & 4 & 1408 & 107 & 2.0 $\times$ $10^{-5}$ & 64 & 2.5 $\times$ $10^{-4}$ & 72 & 1.2 $\times$ $10^{-4}$ & 88 &1.3 $\times$ $10^{-5}$ & 123 & 8.1 $\times$ $10^{-5}$ \\
 1 & 4 & 2816 & 105 & 3.9 $\times$ $10^{-6}$ & 63 & 5.7 $\times$ $10^{-5}$ & 72 & 1.7 $\times$ $10^{-5}$  & 91 &4.0 $\times$ $10^{-6}$ & 126 & 3.4 $\times$ $10^{-6}$ \\
\hline
\end{tabular}
}
\caption{\label{resultsMU3} Scattering experiments for the U-shape
  geometry with $\rho=k_1^2/k_2^2$, and for the CFIEFK$^2$, CFIESK, SCFIE, and CFIERPS formulations In the SCFIE formulation we selected $\eta=k_1$. In the regularized formulations CFIER and CFIERPS we used $\kappa=(k_1+k_2)/2+i\ k_1 $. The GMRES residual was set to equal $10^{-12}$.}
\end{center}
\end{table}

In Table~\ref{resultsSP1} we present computational times required by a matrix-vector product for each of the formulations CFIEFK, CFIESK, SCFIE, CFIER, and CFIERPS. The computational times presented were delivered by a MATLAB implementation of the Nystr\"om
discretization on a MacBookPro machine with $2\times 2.3$ GHz
Quad-core Intel i7 with 16 GB of memory. We present computational
times for the square geometry, as the computational times required by
the U-shaped geometry considered in this text are extremely similar to
those for the square geometry at the same levels of discretization. As
it can be seen from the results in Table~\ref{resultsSP1}, the
computational times required by a matrix-vector product for the
CFIEFK, CFIESK, SCFIE, and CFIERPS formulations are quite similar, while the computational times required by a matrix-vector product related to the CFIER
formulation are on average $1.16$ times more expensive than
those required by the other three formulations. 

\begin{table}
\begin{center}
\resizebox{!}{1.0cm}
{
\begin{tabular}{|c|c|c|c|c|c|}
\hline
 Unknowns & CFIEFK & CFIESK & SCFIE$^*$ & CFIER & CFIERPS\\
\cline{2-6}
\hline
256 & 4.5  & 4.8 &  4.2  & 5.6 & 5.0\\
512 & 15.9 &  16.4 &  15.4 &  19.1 & 17.0\\
1024 & 59.4  & 63.6 &  64.2 &  73.0 & 63.8\\
\hline
\end{tabular}
}
\caption{\label{resultsSP1} Computational times (in seconds) for the matrix-vector products (seconds) needed by the formulations CFIEFK, CFIESK, SCFIE, CFIER, and CFIERPS for the square geometry.}
\end{center}
\end{table}

We present next in Tables~\ref{resultsMS2}--\ref{resultsMU4} the performance of our solvers based on the five formulations considered in this text in the case of high-contrast, high-frequency configurations. We conclude, in conjunction with the results presented in Table~\ref{resultsSP1}, that solvers based on the SCFIE, CFIER, and CFIERPS formulations consistently outperform solvers based on the classical formulations CFIEFK$^2$ and CFIESK in the regime under consideration. Furthermore, the solvers based on the SCFIE and CFIER formulations compare favorably to the solvers based on the CFIESK formulation.

\begin{table}
\begin{center}
\resizebox{!}{1.2cm}
{
\begin{tabular}{|c|c|c|c|c|c|c|c|c|c|c|c|c|}
\hline
$k_1$ & $k_2$ & Unknowns & \multicolumn{2}{c|}{CFIEFK$^2$} &\multicolumn{2}{c|}{CFIESK}&\multicolumn{2}{c|}{SCFIE$^*$}& \multicolumn{2}{c|}{CFIER}&\multicolumn{2}{c|}{CFIERPS}\\
\cline{4-13}
 & & & Iter.& $\epsilon_\infty$ &Iter.&$\epsilon_\infty$&Iter.&$\epsilon_\infty$&Iter.&$\epsilon_\infty$&Iter.&$\epsilon_\infty$\\
\hline
 1 & 4 & 256 & 18 & 6.8 $\times$ $10^{-3}$& 24 & 3.7 $\times$ $10^{-4}$ & 25 & 3.8 $\times$ $10^{-3}$ &26 & 3.1 $\times$ $10^{-3}$& 21 & 3.1 $\times$ $10^{-3}$\\
 2 & 8 & 512 & 24 & 2.9 $\times$ $10^{-3}$ & 39 & 5.0 $\times$ $10^{-4}$ & 37 & 2.1 $\times$ $10^{-4}$ &33 & 3.6 $\times$ $10^{-3}$ & 32 & 1.4 $\times$ $10^{-3}$ \\
 4 & 16 & 1024 & 62 & 4.7 $\times$ $10^{-3}$ & 94 & 4.0 $\times$ $10^{-3}$ & 63 & 1.0 $\times$ $10^{-4}$ &58 & 7.5 $\times$ $10^{-3}$ & 62 & 1.8 $\times$ $10^{-3}$ \\
 8 & 32 & 2048 & 119 & 8.1 $\times$ $10^{-3}$ & 162 & 8.2 $\times$ $10^{-3}$ & 112 & 3.8 $\times$ $10^{-4}$ &102 & 6.6 $\times$ $10^{-3}$ & 115 & 6.7 $\times$ $10^{-3}$ \\
\hline
\end{tabular}
}
\caption{\label{resultsMS2} Scattering experiments for the square
  geometry with $\rho=1$, and for the CFIEFK$^2$, CFIESK, SCFIE, and CFIERPS formulations In the SCFIE formulation we selected $\eta=k_1$. In the regularized formulations CFIER and CFIERPS we used $\kappa=(k_1+k_2)/2+i\ k_1 $. The GMRES residual was set to equal $10^{-4}$.}
\end{center}
\end{table}

\begin{table}
\begin{center}
\resizebox{!}{1.2cm}
{
\begin{tabular}{|c|c|c|c|c|c|c|c|c|c|c|c|c|}
\hline
$k_1$ & $k_2$ & Unknowns & \multicolumn{2}{c|}{CFIEFK$^2$} &\multicolumn{2}{c|}{CFIESK}&\multicolumn{2}{c|}{SCFIE$^*$}& \multicolumn{2}{c|}{CFIER}&\multicolumn{2}{c|}{CFIERPS}\\
\cline{4-13}
 & & & Iter.& $\epsilon_\infty$ &Iter.&$\epsilon_\infty$&Iter.&$\epsilon_\infty$&Iter.&$\epsilon_\infty$&Iter.&$\epsilon_\infty$\\
\hline
 1 & 4 & 352 & 52 & 4.3 $\times$ $10^{-2}$ & 64 & 7.0 $\times$ $10^{-4}$ & 45 & 1.4 $\times$ $10^{-2}$ & 49 & 2.0 $\times$ $10^{-2}$ & 44 & 2.2 $\times$ $10^{-2}$ \\
 2 & 8 & 704 & 76 & 2.5 $\times$ $10^{-2}$ & 107 & 3.2 $\times$ $10^{-3}$ & 78 & 9.4 $\times$ $10^{-4}$ & 79 & 1.5 $\times$ $10^{-2}$ & 75 & 2.0 $\times$ $10^{-2}$ \\
 4 & 16 & 1408 & 117 & 8.7 $\times$ $10^{-3}$ & 149 & 7.7 $\times$ $10^{-3}$ & 136 & 3.3 $\times$ $10^{-4}$ & 124 & 3.9 $\times$ $10^{-3}$ & 113 & 4.4 $\times$ $10^{-3}$ \\
 8 & 32 & 2816 & 281 & 3.3 $\times$ $10^{-2}$ & 351 & 3.2 $\times$ $10^{-2}$ & 257 & 2.8 $\times$ $10^{-4}$ & 257 & 1.4 $\times$ $10^{-2}$ & 244 & 1.2 $\times$ $10^{-2}$ \\
\hline
\end{tabular}
}
\caption{\label{resultsMU2} Scattering experiments for the U-shape
  geometry with $\rho=1$, and for the CFIEFK$^2$, CFIESK, SCFIE, and CFIERPS formulations In the SCFIE formulation we selected $\eta=k_1$. In the regularized formulations CFIER and CFIERPS we used $\kappa=(k_1+k_2)/2+i\ k_1 $. The GMRES residual was set to equal $10^{-4}$.}
\end{center}
\end{table}

\begin{table}
\begin{center}
\resizebox{!}{1.2cm}
{
\begin{tabular}{|c|c|c|c|c|c|c|c|c|c|c|c|c|}
\hline
$k_1$ & $k_2$ & Unknowns & \multicolumn{2}{c|}{CFIEFK$^2$} &\multicolumn{2}{c|}{CFIESK}&\multicolumn{2}{c|}{SCFIE$^*$}& \multicolumn{2}{c|}{CFIER}&\multicolumn{2}{c|}{CFIERPS}\\
\cline{4-13}
 & & & Iter.& $\epsilon_\infty$ &Iter.&$\epsilon_\infty$&Iter.&$\epsilon_\infty$&Iter.&$\epsilon_\infty$&Iter.&$\epsilon_\infty$\\
\hline
 1 & 4 & 256 & 30 & 9.8 $\times$ $10^{-4}$& 23 & 6.7 $\times$ $10^{-5}$ & 28 & 2.1 $\times$ $10^{-3}$ & 33 & 2.1 $\times$ $10^{-4}$& 47 & 4.3 $\times$ $10^{-4}$\\
 2 & 8 & 512 & 53 & 1.2 $\times$ $10^{-3}$ & 34 & 1.3 $\times$ $10^{-4}$ & 59 & 3.6 $\times$ $10^{-4}$ & 39 & 6.5 $\times$ $10^{-4}$ & 72 & 8.5 $\times$ $10^{-4}$ \\
 4 & 16 & 1024 & 82 & 6.7 $\times$ $10^{-4}$ & 53 & 7.0 $\times$ $10^{-3}$ & 88 & 2.5 $\times$ $10^{-4}$ & 51 & 6.7 $\times$ $10^{-4}$ & 99 & 1.8 $\times$ $10^{-3}$ \\
 8 & 32 & 2048 & 236 & 1.6 $\times$ $10^{-3}$ & 112 & 2.1 $\times$ $10^{-4}$ & 205 & 2.3 $\times$ $10^{-4}$ & 111 & 1.8 $\times$ $10^{-3}$ & 197 & 4.3 $\times$ $10^{-3}$ \\
\hline
\end{tabular}
}
\caption{\label{resultsMS4} Scattering experiments for the square
  geometry with $\rho=k_1^2/k_2^2$, and for the CFIEFK$^2$, CFIESK, SCFIE, and CFIERPS formulations In the SCFIE formulation we selected $\eta=k_1$. In the regularized formulations CFIER and CFIERPS we used $\kappa=(k_1+k_2)/2+i\ k_1 $. The GMRES residual was set to equal $10^{-4}$.}
\end{center}
\end{table}

\begin{table}
\begin{center}
\resizebox{!}{1.2cm}
{
\begin{tabular}{|c|c|c|c|c|c|c|c|c|c|c|c|c|}
\hline
$k_1$ & $k_2$ & Unknowns & \multicolumn{2}{c|}{CFIEFK$^2$} &\multicolumn{2}{c|}{CFIESK}&\multicolumn{2}{c|}{SCFIE$^*$}& \multicolumn{2}{c|}{CFIER}&\multicolumn{2}{c|}{CFIERPS}\\
\cline{4-13}
 & & & Iter.& $\epsilon_\infty$ &Iter.&$\epsilon_\infty$&Iter.&$\epsilon_\infty$&Iter.&$\epsilon_\infty$&Iter.&$\epsilon_\infty$\\
\hline
 1 & 4 & 352 & 76 & 5.4 $\times$ $10^{-4}$ & 40 & 4.8 $\times$ $10^{-3}$ & 57 & 5.5 $\times$ $10^{-3}$ & 58 & 5.1 $\times$ $10^{-4}$ & 66 & 3.5 $\times$ $10^{-4}$ \\
 2 & 8 & 704 & 98 & 7.7 $\times$ $10^{-4}$ & 64 & 2.0 $\times$ $10^{-3}$ & 87 & 9.6 $\times$ $10^{-4}$ & 66 & 1.9 $\times$ $10^{-4}$ & 107 & 8.1 $\times$ $10^{-4}$ \\
 4 & 16 & 1408 & 236 & 1.7 $\times$ $10^{-3}$ & 126 & 2.2 $\times$ $10^{-3}$ & 168 & 4.3 $\times$ $10^{-4}$ & 128 & 9.2 $\times$ $10^{-4}$ & 181 & 2.5 $\times$ $10^{-3}$ \\
 8 & 32 & 2816 & 424 & 3.2 $\times$ $10^{-3}$ & 252 & 2.8 $\times$ $10^{-3}$ & 286 & 7.0 $\times$ $10^{-4}$ & 216 & 1.1 $\times$ $10^{-3}$ & 305 & 3.3 $\times$ $10^{-3}$ \\
\hline
\end{tabular}
}
\caption{\label{resultsMU4} Scattering experiments for the U-shape
  geometry with $\rho=k_1^2/k_2^2$, and for the CFIEFK$^2$, CFIESK, SCFIE, and CFIERPS formulations In the SCFIE formulation we selected $\eta=k_1$. In the regularized formulations CFIER and CFIERPS we used $\kappa=(k_1+k_2)/2+i\ k_1 $. The GMRES residual was set to equal $10^{-4}$.}
\end{center}
\end{table}

We conclude with an illustration in Tables~\ref{resultsMS2R}--\ref{resultsMU2R} of high-contrast high-frequency scenarios whereby the computational gains associated with solvers based on the SCFIE, CFIER, and CFIERPS are the most significant. As it can be seen in Table~\ref{resultsMU2R}, solvers based on the SCFIE formulations and CFIER formulations can result in computational times that are at least $3$ times faster than those based on the classical CFIESK and CFIEFK formulations.

\begin{table}
\begin{center}
\resizebox{!}{1.2cm}
{
\begin{tabular}{|c|c|c|c|c|c|c|c|c|c|c|c|c|}
\hline
$k_1$ & $k_2$ & Unknowns & \multicolumn{2}{c|}{CFIEFK$^2$} &\multicolumn{2}{c|}{CFIESK}&\multicolumn{2}{c|}{SCFIE$^*$}& \multicolumn{2}{c|}{CFIER}&\multicolumn{2}{c|}{CFIERPS}\\
\cline{4-13}
 & & & Iter.& $\epsilon_\infty$ &Iter.&$\epsilon_\infty$&Iter.&$\epsilon_\infty$&Iter.&$\epsilon_\infty$&Iter.&$\epsilon_\infty$\\
\hline
 3.5 & 1 & 256 & 16 & 2.1 $\times$ $10^{-3}$& 23 & 8.1 $\times$ $10^{-4}$ & 21 & 1.1 $\times$ $10^{-3}$ &20 & 1.7 $\times$ $10^{-3}$& 21 & 3.3 $\times$ $10^{-3}$\\
 7 & 2 & 512 & 29 & 1.7 $\times$ $10^{-3}$ & 41 & 2.0 $\times$ $10^{-3}$ & 24 & 4.4 $\times$ $10^{-4}$ & 20 & 2.1 $\times$ $10^{-3}$ & 30 & 2.0 $\times$ $10^{-3}$ \\
 14 & 4 & 1024 & 59 & 8.9 $\times$ $10^{-3}$ & 56 & 1.7 $\times$ $10^{-1}$ & 35 & 3.2 $\times$ $10^{-4}$ & 22 & 1.1 $\times$ $10^{-3}$ & 57 & 4.2 $\times$ $10^{-3}$ \\
 28 & 8 & 2048 & 85 & 4.1 $\times$ $10^{-2}$ & 94 & 9.4 $\times$ $10^{-2}$ & 39 & 5.5 $\times$ $10^{-4}$ & 25 & 1.1 $\times$ $10^{-3}$ & 87 & 1.2 $\times$ $10^{-2}$ \\
\hline
\end{tabular}
}
\caption{\label{resultsMS2R} Scattering experiments for the square
  geometry with $\rho=1$, and for the CFIEFK$^2$, CFIESK, SCFIE, and CFIERPS formulations In the SCFIE formulation we selected $\eta=k_1$. In the regularized formulations CFIER and CFIERPS we used $\kappa=(k_1+k_2)/2+i\ 4 $. The GMRES residual was set to equal $10^{-4}$.}
\end{center}
\end{table}

\begin{table}
\begin{center}
\resizebox{!}{1.2cm}
{
\begin{tabular}{|c|c|c|c|c|c|c|c|c|c|c|c|c|}
\hline
$k_1$ & $k_2$ & Unknowns & \multicolumn{2}{c|}{CFIEFK$^2$} &\multicolumn{2}{c|}{CFIESK}&\multicolumn{2}{c|}{SCFIE$^*$}& \multicolumn{2}{c|}{CFIER}&\multicolumn{2}{c|}{CFIERPS}\\
\cline{4-13}
 & & & Iter.& $\epsilon_\infty$ &Iter.&$\epsilon_\infty$&Iter.&$\epsilon_\infty$&Iter.&$\epsilon_\infty$&Iter.&$\epsilon_\infty$\\
\hline
 3.5 & 1 & 352 & 42 & 3.4 $\times$ $10^{-3}$ & 51 & 2.4 $\times$ $10^{-3}$ & 30 & 3.6 $\times$ $10^{-3}$ & 26 & 1.5 $\times$ $10^{-3}$ & 37 & 2.4 $\times$ $10^{-3}$ \\
 7 & 2 & 704 & 54 & 3.2 $\times$ $10^{-3}$ & 70 & 2.7 $\times$ $10^{-2}$ & 35 & 3.0 $\times$ $10^{-4}$ & 33 & 1.3 $\times$ $10^{-3}$ & 47 & 2.7 $\times$ $10^{-3}$ \\
 14 & 4 & 1408 & 123 & 2.7 $\times$ $10^{-2}$ & 148 & 9.5 $\times$ $10^{-2}$ & 47 & 3.2 $\times$ $10^{-4}$ & 46 & 1.2 $\times$ $10^{-3}$ & 77 & 2.6 $\times$ $10^{-3}$ \\
 28 & 8 & 2816 & 238 & 9.8 $\times$ $10^{-2}$ & 240 & 1.3 $\times$ $10^{-1}$ & 84 & 3.6 $\times$ $10^{-4}$ & 67 & 1.8 $\times$ $10^{-3}$ & 169 & 4.2 $\times$ $10^{-3}$ \\
\hline
\end{tabular}
}
\caption{\label{resultsMU2R} Scattering experiments for the U-shape
  geometry with $\rho=1$, and for the CFIEFK$^2$, CFIESK, SCFIE, and CFIERPS formulations In the SCFIE formulation we selected $\eta=k_1$. In the regularized formulations CFIER and CFIERPS we used $\kappa=(k_1+k_2)/2+i\ 4 $. The GMRES residual was set to equal $10^{-4}$.}
\end{center}
\end{table}
 
 Finally, we present in Table~\ref{resultsMS2Rounded} a comparison between scattering solutions corresponding to Lipschitz scatterers and solutions corresponding to nearby smooth scatterers that are obtained from rounding the corners. More specifically, we considered the sphere of radius 2 in $\mathbb{R}^2$ using the $\ell^q$ norm for $q=512$, that is
  \[
  B_q:=\{(x_1,x_2)\in\mathbb{R}^2: x_1^q+x_2^q = 2^q\},\ \quad q=512
\]
which is a close and smooth (rounded) approximation of the square geometry in Figure~\ref{fig:U}---indeed, the distance between the scatterers is about $10^{-3}$. We compare the far-field signature of $B_q, q=512$ to that of the square for various wavenumbers using the various boundary integral equation formulations considered in this text. We note that for a given frequency the numbers of iterations required by boundary integral formulations to reach the same tolerance are roughly the same for the square and its very close smooth (rounded) approximation---see Tables~\ref{resultsMS2} and~\ref{resultsMS2Rounded}. 
\begin{table}
\begin{center}
\resizebox{!}{1.2cm}
{
\begin{tabular}{|c|c|c|c|c|c|c|c|c|c|c|c|c|}
\hline
$k_1$ & $k_2$ & Unknowns & \multicolumn{2}{c|}{CFIEFK$^2$} &\multicolumn{2}{c|}{CFIESK}&\multicolumn{2}{c|}{SCFIE$^*$}& \multicolumn{2}{c|}{CFIER}&\multicolumn{2}{c|}{CFIERPS}\\
\cline{4-13}
 & & & Iter.& $\epsilon_\infty$ &Iter.&$\epsilon_\infty$&Iter.&$\epsilon_\infty$&Iter.&$\epsilon_\infty$&Iter.&$\epsilon_\infty$\\
\hline
 1 & 4 & 256 & 20 & 5.0 $\times$ $10^{-3}$& 25 & 2.7 $\times$ $10^{-4}$ & 19 & 8.7 $\times$ $10^{-3}$ &30 & 5.8 $\times$ $10^{-3}$& 21 & 5.8 $\times$ $10^{-3}$\\
 2 & 8 & 512 & 24 & 2.8 $\times$ $10^{-3}$ & 41 & 3.7 $\times$ $10^{-4}$ & 30 & 1.4 $\times$ $10^{-3}$ &32 & 5.6 $\times$ $10^{-3}$ & 32 & 1.7 $\times$ $10^{-3}$ \\
 4 & 16 & 1024 & 62 & 4.0 $\times$ $10^{-3}$ & 97 & 1.8 $\times$ $10^{-3}$ & 55 & 1.4 $\times$ $10^{-3}$ &59 & 5.6 $\times$ $10^{-3}$ & 62 & 1.8 $\times$ $10^{-3}$ \\
 8 & 32 & 2048 & 122 & 7.9 $\times$ $10^{-3}$ & 173 & 5.6 $\times$ $10^{-3}$ & 96 & 3.3 $\times$ $10^{-3}$ &103 & 8.7 $\times$ $10^{-3}$ & 117 & 4.6 $\times$ $10^{-3}$ \\
\hline
\end{tabular}
}
\caption{\label{resultsMS2Rounded} Scattering experiments for the $B_q,\ q=512$ sphere of radius 2 with $\rho=1$, and for the CFIEFK$^2$, CFIESK, SCFIE, and CFIERPS formulations In the SCFIE formulation we selected $\eta=k_1$. In the regularized formulations CFIER and CFIERPS we used $\kappa=(k_1+k_2)/2+i\ k_1 $. The GMRES residual was set to equal $10^{-4}$.}
\end{center}
\end{table} 

\section{Conclusions}

In this work we have presented high-order  Nystr\"om  discretizations based on polynomially graded meshes for several boundary integral formulations including certain regularized formulations for Helmholtz transmission problems in domains with corners. We have rigorously proven  the well-posedness of some of these formulations and have shown that solvers based on the regularized and on the single integral equations outperform solvers based on commonly used boundary integral equation formulations in the case of high-contrast, high-frequency applications. The numerical analysis of these schemes will be subject of future investigation.  Extensions of the regularization scheme used in this paper to the case of multiple dielectric scatterers will also be subject of future investigation. 

\section*{Acknowledgments}
Catalin Turc gratefully  acknowledges  support from NSF through contract DMS-1312169. V\'{\i}ctor Dom\'{i}nguez is partially supported by Ministerio de Econom\'{\i}a y Competitividad, through the grant  MTM2014-52859. Part of this research was carried out during a short visit of V\'{\i}ctor Dom\'{\i}nguez to NJIT. 

We want to thank the reviewer for his/her very careful reading and useful comments which help us to improve the readability  and to complement the theoretical results of this paper.
\bibliography{biblio}

\end{document}